\documentclass[11pt]{amsart}
\usepackage{url}
\usepackage[dvips]{hyperref}%needs to be last package listed; use dvips -z.
\input epsf

\newtheorem{Thm}{Theorem}
\newtheorem{Lem}{Lemma}

\newcommand{\dia}{$\diamondsuit$ }
\newcommand{\diaa}{$\diamondsuit\!$}
\newcommand{\diaas}{$\diamondsuit$\hskip -0.3 pt s }

\newcommand{\whboxx}{$\square$\hskip -0.2 pt}

\begin{document}
%\date{(date1), and in revised form (date2).}
\thanks{Work on this material was done at home and at Wolfson College, Oxford (2005--2006). I am grateful for Oxford hospitality and Winnipeg patience.}

\title[Multicoloured Isonemal Fabrics by Thin Striping]{Colouring Isonemal Fabrics with more than two Colours by Thin Striping}

\author{Robert S.~D.~Thomas}

\address{St John's College and Department of Mathematics, University of Manitoba, Winnipeg, Manitoba  R3T 2N2  Canada}

\email{thomas@cc.umanitoba.ca}

\begin{abstract}
Perfect colouring of isonemal fabrics by thin striping of warp and weft with more than two colours is examined. Examples of thin striping in all possible species with no redundancy and with redundant cells arranged as twills are given. Colouring woven flat tori is discussed.
\end{abstract}

\maketitle

\section{Background}%1
\label{sect:Backg}

\noindent This paper takes up the subject of colouring fabrics perfectly with more than two colours introduced in \cite{Thomas2013} and extends it to thin striping, that is, to colouring the strands in both directions with a finite number of colours cyclically. The first three sections of \cite{Thomas2013} and the standard references \cite{GS1980,GS1988} are relevant and will not be repeated here.

At a referee's request three paragraphs will give some idea of terms and notations to be used. A {\em fabric} is a two-fold weave that does not fall apart having {\em warps} (vertical and dark in colour) crossing {\em wefts} (horizontal and pale) in square {\em cells}, in each of which one {\em strand} (warp or weft) is uppermost or visible in an illustration of the {\em obverse} side of the weave. The other side, the {\em reverse}, is illustrated (if at all) as reflected in a mirror set up behind/below the plane of the fabric to preserve symmetries. {\em Isonemal} fabrics have symmetry group $G_1$ transitive on the strands. The pattern of over and under is periodic along a strand, this period being called {\em order}, distinct from the two-dimensional {\em period}, the area of the smallest period parallelograms. A group's period parallelogram with corners that are translates of one another is called a {\em lattice unit}. Two catalogues of fabrics have been published \cite{GS1985,GS1986} giving catalogue numbers for isonemal fabrics of small order in the form $a$-$b$-$c$, where $a$ is order, $b$ is a unique representation of the over-and-under pattern, and $c$ is a serial number.

Most isonemal fabrics were divided into five {\em genera} by Gr\"unbaum and Shephard \cite{GS1985} and 39 {\em species} in \cite{Thomas2009,Thomas2010a,Thomas2010b}, based on Roth's \cite{Roth1993} description of 39 {\em types} of symmetry-group pairs $\langle G_1, H_1\rangle$, where $H_1$ is the subgroup of $G_1$ consisting of elements not including reflection in the plane of the fabric, represented $\tau$. (The few fabrics outside a genus or species are called {\em exceptional}.) $G_1$ consists of elements $\langle s,t\rangle$, where $s$ is a two-dimensional transformation (translation, reflection, glide-reflection, half-turn (\!\diaa), or quarter-turn (\whboxx)) and $t$ is either the identity $e$ or the reflection $\tau$. $H_1$ is the subgroup consisting of elements $\langle s,e\rangle$, and $G_2$ is the group that is the projection of $G_1$ defined by $\langle s,\tau\rangle\mapsto\langle s,e\rangle$ and $\langle s,e\rangle\mapsto\langle s,e\rangle$ to avoid the issue of reflections in the plane of the fabric. Axes of reflection pass diagonally across cells (Fig.~\ref{fig:8dea}); axes of glide-reflection are not so restricted (Fig.~\ref{fig:8dea}, Fig.~\ref{fig:9ab}a). When they do so (Fig.~\ref{fig:10ab}a), they are said to be in {\em mirror position}. Figure~\ref{fig:8dea}ab illustrates the simplest isonemal fabrics, {\em twills}, in which each strand's pattern is the adjacent strand's shifted by one; a twill's over-and-under pattern is represented $a_1/a_2/b_1/b_2/\dots$ until its order, the sum of $a_1 + a_2 + \dots$, is used up.
\begin{figure}%1, 2 before
\[
\begin{array}{ccc}
\noindent\epsffile{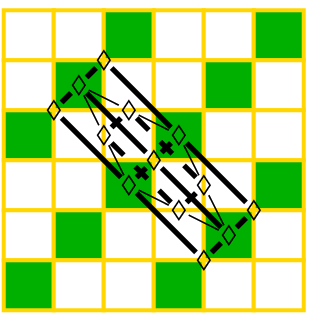} &\epsffile{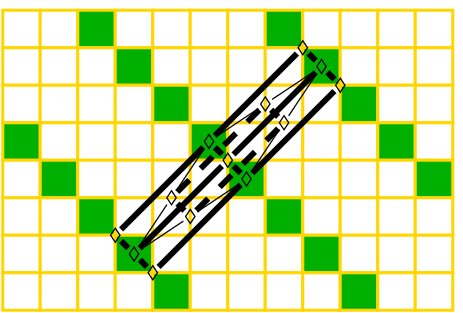} &\epsffile{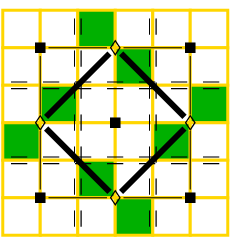}
\\%40, 30%
\mbox{(a)} &\mbox{(b)} &\mbox{(c)}
\end{array}
\]
\caption{Designs of fabrics that could be patterns of redundancy with symmetries marked. a. The 2/1 twill. b. The 4/1 twill. c. 4-1-1.}
\label{fig:8dea}
\end{figure}

The {\em standard} colouring (warp dark, weft pale) is a {\em perfect colouring} in the sense that each symmetry in $G_1$ either preserves or reverses the colour of all cells. A non-standard colouring is called {\em perfect} if each symmetry permutes the colours of all cells of each colour. When the two strands passing through a cell are of the same colour, the cell is called {\em redundant} in that colouring because which strand is uppermost has no effect on the cell's appearance.
\begin{figure}%2, 3 before
\[
\begin{array}{cc}
\noindent
\epsffile{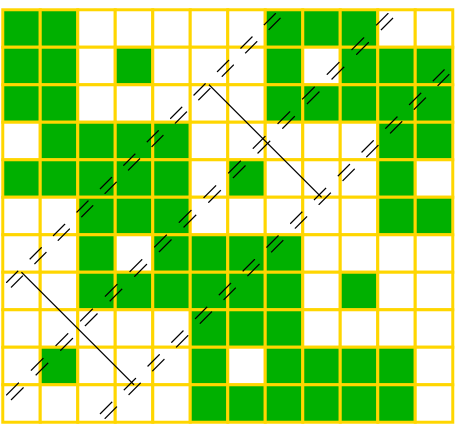}&\epsffile{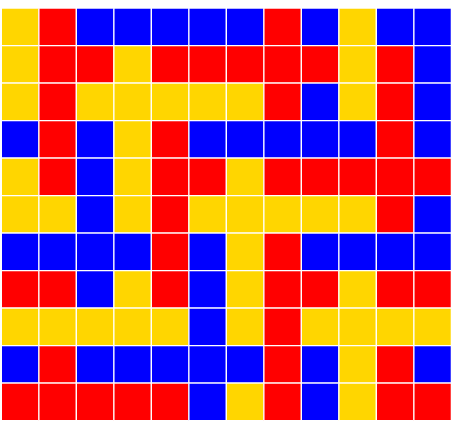}\\%30%[3-col. p. 2 top; ]
\mbox{(a)} &\mbox{(b)}
\end{array}
\]
\caption{a. Order-30 species-$1_o$ example of Figure 10b of \cite{Thomas2009}. b. Three-colouring by thin striping with redundant cells strictly between the cells through which the axes run.}
\label{fig:9ab}
\end{figure}
\begin{figure}%3, 4 before
\[
\begin{array}{cc}
\epsffile{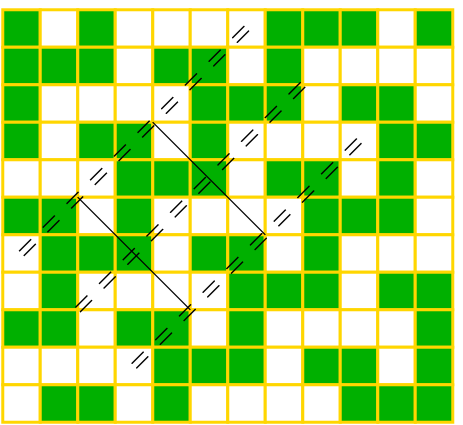} &\epsffile{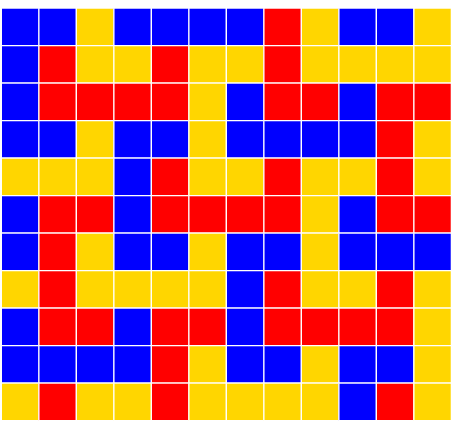}\\%30%[Roth's 12-183-1parallel 4 in 2003 MS p. 2 between 6 and 7---bottom C's ]
\mbox{(a)} &\mbox{(b)}
\end{array}
\]
\caption{a. Design of Roth's \cite{Roth1993} example 12-183-1 of species $1_m$. b. Three-colouring by thin striping with redundant cells along the axes.}
\label{fig:10ab}
\end{figure}

Thin stripes overlap in individual cells---all cells---of the fabric.
Thin stripes of the same colour overlap in redundant cells, in this paper arranged as twills.
The natural unit of measurement in the fabric is the side of a cell, in which the diagonal of the square cells have the length $\sqrt 2$, which will be abbreviated $\delta$. 
It is convenient to represent $\delta/2$ by $\beta$.
In \cite{Thomas2013} the following results about thin striping were shown.

\medskip
\begin{Thm} %1
Some non-exceptional fabrics in all species can be perfectly colour\-ed by the assignment of different colours to all strands.
\label{thm:1}
\end{Thm}

\begin{Thm} %2
If no colour is common to warps and wefts in the assignment of a finite number of colours to the striping of an isonemal fabric, then the necessary conditions:
the same number of colours are assigned to warps as to wefts, and
the vertical and horizontal sequences of the colours of the stripes are periodic with the same period,
are sufficient constraints on the colours to allow perfect colouring by thin striping.
\label{thm:2}
\end{Thm}

\section{No Redundancy}%
\label{sect:noRed}

\noindent Sections 5 and 6 of \cite{Thomas2013} were written on the erroneous assumption, conflicting with earlier sections of that paper, that there will always be redundant cells in a colouring.%
\footnote{This apparently absurd event is an inexcusable result of sections 5 and 6 having been written when two-colourings were generalized to more than two colours with redundancy before the further generalization of earlier sections of the paper had been thought of.}
That is of course only the case when there are perpendicular strands of the same colour. This section, which should have been in \cite{Thomas2013}, takes up the interesting possibility where there are no such strands and so no redundancy, there being stripes of two colours in each direction, a total of four colours --- not too many to make an attractive pattern.
\begin{figure}%4
\[\begin{array}{cc}
\epsffile{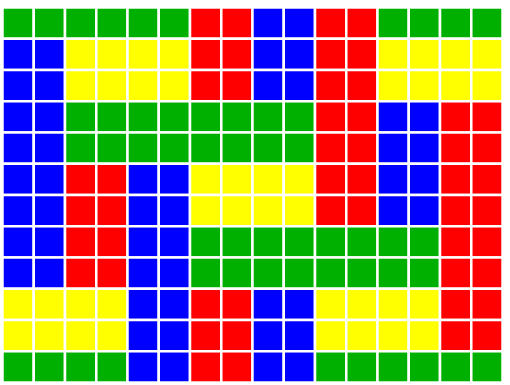} &\epsffile{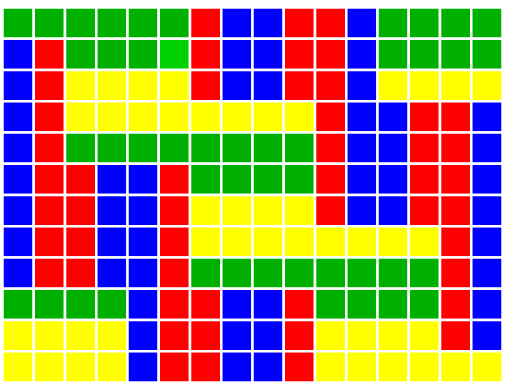}\\%25%
\mbox{(a)} &\mbox{(b)}
\end{array}\]
\caption{Fabric of species $11_e$ in Fig.~\ref{fig:40a}b thickly 4-coloured two ways with no redundant cells.}\label{fig:nr12}.
\end{figure}
\begin{figure}%5
\[\begin{array}{cc}
\epsffile{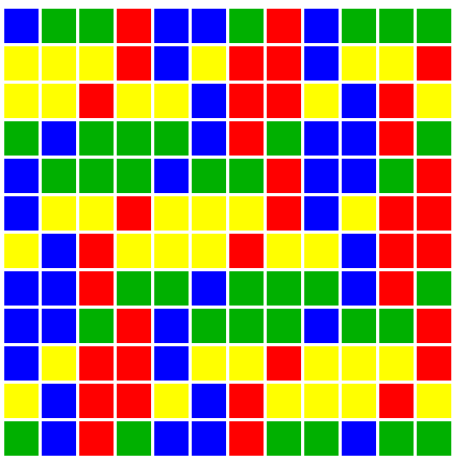} &\epsffile{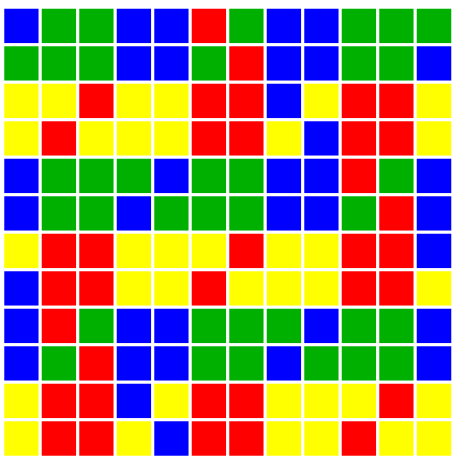}\\%30%
\mbox{(a)} &\mbox{(b)}
\end{array}\]
\caption{Fabric 12-315-4 of species 30 in Fig.~\ref{fig:13ab}a thickly 4-coloured two ways with no redundant cells.}\label{fig:nr34}.
\end{figure}
\begin{figure}%6
\[\begin{array}{ccc}
\epsffile{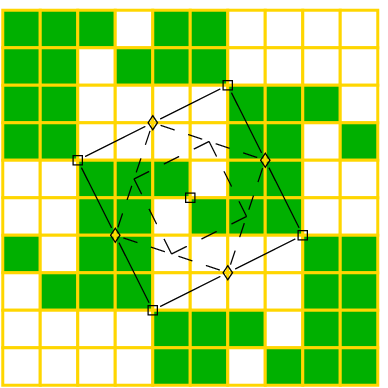} &\epsffile{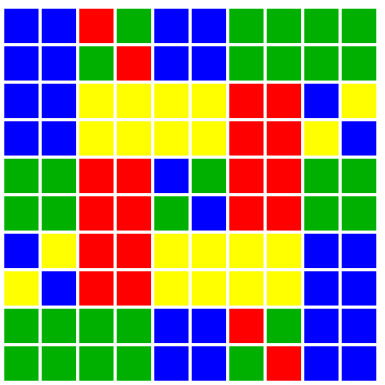} &\epsffile{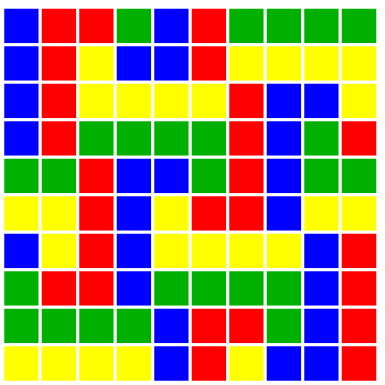}\\%30%
\mbox{(a)} &\mbox{(b)} &\mbox{(c)}
\end{array}\]
\caption{a. Fabric 10-55-2 of species $33_3$ (Fig.~5a of \cite{Thomas2013}). b, c. Thickly 4-coloured two ways with no redundant cells.}\label{fig:nr56}.
\end{figure}

Lemma 3.1 of \cite{Thomas2013} gives the necessary and sufficient condition for an assignment of all different colours to a thick striping of a fabric to be perfect, namely preservation of the stripes by the symmetries, followed by the conditions on the symmetries to allow this. No mention is made of conditions on the striping, but striping almost always has to be done with some care, as was noted by Roth \cite{Roth1995}. In particular, the striping of wefts dictates how the warps can be striped, as is illustrated with four colours in Fig.~\ref{fig:nr12}a and b. These diagrams show two ways in which the fabric of Fig.~\ref{fig:40a}b can be thickly striped perfectly with two colours in each direction. In each case, the striping in one direction dictates the boundaries of the stripes in the other direction. What one cannot do is to stripe the fabric with its horizontal stripes as in Fig.~\ref{fig:nr12}a and its vertical stripes as in Fig.~\ref{fig:nr12}b. Not that this cannot be woven, but warp-weft-interchanging symmetries cannot take stripes to stripes. The same condition of Lemma 3.1 suffices for the same finite number of colours in each direction repeating cyclically with no redundancy. While this condition is a slight relaxation of the conditions for striping with two colours --- and therefore unavoidably with redundancy, it still bans twills and twillins as in Section 3 of \cite{Thomas2012}. That eliminates species 1--10, 12, 14, 16, 18, 20, 23, 24, 26, 28, 31, 32, 34, 36, and 39 leaving 11, 13, 15, 17, 19, 21, 22, 25, 27${}_e$, 29, 30, 33, 35, 37, and 38. Any fabric of these species should be perfectly thickly colorable with any number of colours greater than 1 in each direction. Only two colours will be explored briefly because they are likely to be more attractive than more colours. Examples are in Figures~\ref{fig:nr34} and \ref{fig:nr56} as well as \ref{fig:nr12}. These are all from species that cannot be thickly 4-coloured with twilly redundancy. That they cannot be thickly 4-coloured with twilly redundancy led to the erroneous remark in \cite{Thomas2013} (which assumed, as explained in footnote 1, that redundancy was inevitable) that fabrics of species 38 (p.~50) and 33--39 (p.~54) cannot be thickly coloured with four colours at all. Theorems~\ref{thm:3} and \ref{thm:4} clarify this for thin striping.
\begin{figure}%7
\[\begin{array}{cc}
\epsffile{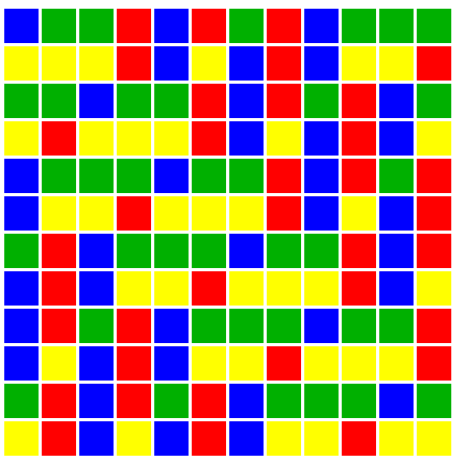} &\epsffile{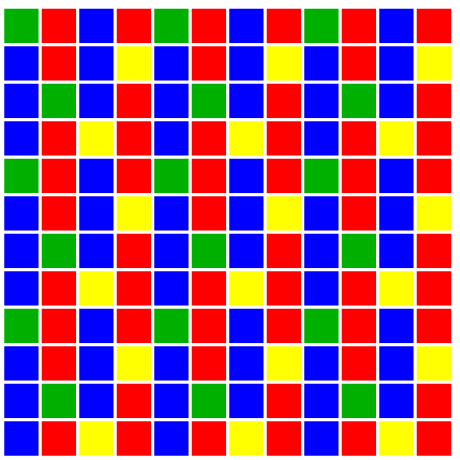}\\%25%
\mbox{(a)} &\mbox{(b)}
\end{array}\]
\caption{Thin 4-colourings with no redundant cells. a. Fabric of species 30 thickly coloured in Fig.~\ref{fig:nr34}. b. 4-1-1 reverse.}\label{fig:nr8}.
\end{figure}
\begin{figure}%8
\[\begin{array}{cc}
\epsffile{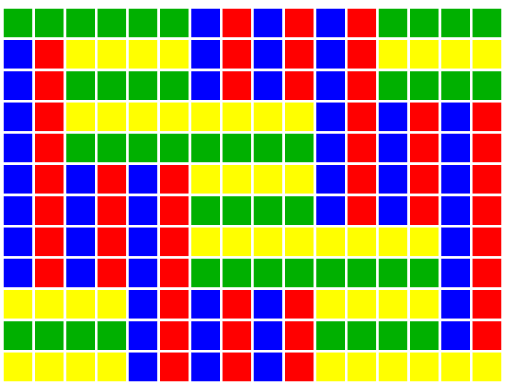} &\epsffile{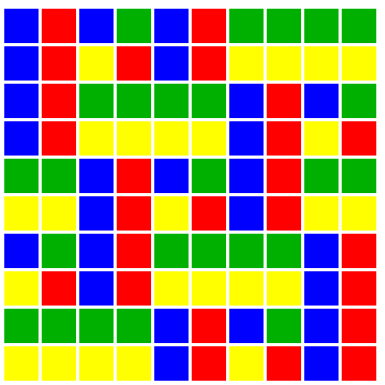}\\%30%
\mbox{(a)} &\mbox{(b)}
\end{array}\]
\caption{Fabrics of species $11_e$ and $33_3$, thickly coloured in Figg.~\ref{fig:nr12} and \ref{fig:nr56}, thinly 4-coloured with no redundant cells.}\label{fig:nr79}.
\end{figure}

Turning to thin striping, one sees that none of the thick-striping restrictions  on species is relevant. And any striping in both directions with no redundancy will do; what can go wrong with thick stripes mentioned in the previous paragraph cannot occur with thin stripes. Unfortunately, since what can be coloured that way is just longer or shorter strips, no interesting motif of a single colour can emerge, and two-coloured motifs seem not to be of much interest.%
\footnote{Not of sufficient interest to be noticed? This may be a matter of psychology or eyesight. When I look at Figure~\ref{fig:48a}b I recognize immediately a red and blue braid motif but must search for the green and yellow one that perfect colouring ensures. A possible exception to this claim is the smallest possible two-colour motif, just two cells, illustrated in Figure~\ref{fig:nr8}b.}
As examples I give a thinly striped version of the three fabrics used to illustrate thick striping (there not being two essentially different thin stripings) in Figures~\ref{fig:nr8} and \ref{fig:nr79}.

\section{Preliminaries}%
\label{sect:prelim}

\noindent The remainder of this paper determines the constraints on the fabric designs to allow perfect colourings by thin striping with twilly redundancy and how to arrange them. In \cite{Thomas2013} the following result is Lemma 4.4.

\begin{Lem} %1
The assignment of a finite number of colours to stripes of an isonemal fabric can be a perfect colouring only if the colours of the warps are the colours of the wefts or no colour is shared.
\label{lem:1}
\end{Lem}

The order of the colours does not need to be the same vertically as horizontally. 
For example, Figure~\ref{fig:69ab} shows a five-colouring otherwise of Roth's \cite{Roth1993} example of species 34, catalogue number 10-107-1.
This topic will not be pursued here.
\begin{figure}%9, 1 before
\[
\begin{array}{cc}
\epsffile{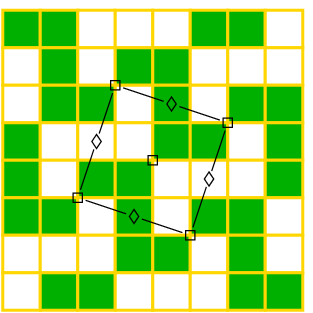} &\epsffile{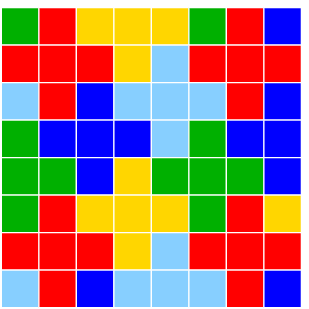}\\%30% middle p. 1

\mbox{(a)} &\mbox{(b)}
\end{array}
\]
\caption{a. Roth's \cite{Roth1993} example of species $34$, 10-107-1 with lattice unit and centres of $G_1$ marked (Fig.~8a of \cite{Thomas2010b}). b. Five-colouring by thin striping.}
\label{fig:69ab}
\end{figure}
The centres of the cross motifs are redundant cells at which strands of the same colour cross.
In order for the colouring to be a perfect colouring, Roth's Colouring Theorem \cite{Roth1995} shows that the symmetry group $G_1$ of the fabric must leave the redundant cells/blocks invariant as a whole and each element of the group must permute them, if at all, in a consistent way. 
The group $G_2$ of the fabric must be a subgroup of the group $G_2$ of the pattern of redundant cells, as in Figures~\ref{fig:8dea}a and b.
For a specific colouring of the wefts, say, as illustrated in Figure~\ref{fig:69ab}b, the positions of the redundant cells must be chosen in accordance with the consistency requirement.
This choice determines the colouring of the warps.

From now on, warps and wefts will be coloured with the same finite set of $c$ colours. For thin striping, the redundant cells may be arranged as $(c-1)/1$ twills, and for $c\geq 4$ as various other isonemal fabrics, 6-1-1, 8-1-1, 8-1-2, and including in particular the square satins beginning with 5-1-1.  Fabrics will have order five or more, and redundant cells will be arranged as twills (Figg.~\ref{fig:8dea}a and b).

There is reason to consider mostly isomenal fabrics with order more than four because those of order four and less are few and mostly exceptional, but in order to consider colourings with three colours, one needs --- there being no choice --- to begin with the 2/1 twill (Figure~\ref{fig:8dea}a) as a redundancy pattern.
Mainly because of their aesthetic appeal, but also for practical reasons, I intend to use three colours to represent what can be done with a twilly pattern of redundant cells for any odd number of colours, four colours for a twilly redundancy pattern for any even number of colours congruent to 0 mod 4, and six colours for a twilly redundancy pattern for even numbers of colours congruent to 2 mod 4.
Also, because side-reversing symmetries (mirrors and side-reversing glide-reflections, half-turns, and quarter-turns) of a fabric relate its opposite sides and one sees fabrics one side at a time, side-reversing symmetries have less appeal in striped fabrics.
Emphasis will be on fabrics with side-preserving symmetries just because their symmetries are visible when the fabric is striped.%
\footnote{While the presence of side-reversing symmetry makes it uninteresting and therefore unnecessary to look at both sides of such coloured fabrics, the absence of all such symmetry makes the two sides of a fabric that lacks it patterns of independent interest.}

\begin{Thm}%3-STRAND
For three colours, perfect colouring of non-exceptional fabrics can be done only with a twilly arrangement of redundant cells.
\label{thm:3}
\end{Thm}
\begin{proof} The same colours must be used for warp and weft because three is odd. The only isonemal design to serve as a redundancy pattern with order three and a single dark cell per order length is the 2/1 twill.
\end{proof}
Here three strands are not standing proxy, as they will later, for larger odd numbers. Similarly, four strands in the next theorem are not standing proxy for more. These two numbers of colours, like two, are exceptional.
\begin{Thm}%4-STRAND
For four colours, perfect colouring of non-exceptional fabrics can be done either with no redundancy or with a twilly arrangement of redundant cells.
\label{thm:4}
\end{Thm}
\begin{proof} Two colours in each direction will colour with no redundancy. With redundancy, there is the 3/1 twill and the apparent exception, 4-1-1 (Figure~\ref{fig:8dea}c), which must be shown to be useless. This design is exceptional both in being of order four and in being a member of no genus.

It is impossible for the group $G_2$ of an isonemal fabric to be a subgroup of the group $G_2$ of 4-1-1 for several reasons.
The glide-reflection axes of 4-1-1 are horizontal and vertical, and those of other isonemal fabrics are all at $\pi/4$ to the horizontal.
The mirrors of 4-1-1 are at the right angle to be potentially useful, but they are $2\delta$ apart. 
Any selection from them (for genus I or II) will be an even number of $\delta$s apart, and the isonemality constraint of \cite{Thomas2009} prevents such a selection from combining with a diagonal period of 4$\delta$ (or multiple of 4$\delta$).
Finally, the centres of rotation, as the diagram indicates, fall only on alternate strand boundaries, preventing an isonemal fabric from falling into genus III, IV, or V.
\end{proof}

This does not say that 4-1-1 cannot be the pattern of redundant cells for a thin striping, only that the only isonemal fabric for which it is useful is itself. 
Unfortunately, when a design that is suitable as a redundancy pattern is itself coloured, the result is just stripes.
Compare Figure~\ref{fig:nr8}b, 4-1-1 4-coloured with no redundancy.
All that changes with a different redundancy pattern is the order of the stripes.
With 4-1-1 the order of the colours of the stripes is different in the two directions, but with a twill as redundancy pattern the order is the same in the two directions.
It is obvious in twilly cases that no fabric with quarter-turn symmetry can have $G_2$ a subgroup of the 2/1 or 3/1 twill---or of any other $(c-1)/1$ twill doubled or not, since these twills fall into species with no quarter-turns in their $G_2$s.
(2/1 is of $28_o$, 3/1 is of $26_e$, 6-3-1 is of $27_o$, and 8-3-1 is of subspecies $25_e$.)
Accordingly, no fabric of species 33--39 can be perfectly coloured thickly or thinly with the same three or four colours vertically and horizontally --- or, with twilly redundancy, any number of colours greater than two.
Call this the quarter-turn ban.
On the other hand, some fabrics from all other species are perfectly 3-colourable by thin striping.

\section{Three colours}%2
\label{sect:3C}

\noindent The $G_1$ of the redundant twill 2/1 for thin striping with three colours, of subspecies $28_o$ and illustrated in Figure~\ref{fig:8dea}a, contains perpendicular glide-reflections with axes not in mirror position and mirrors, allowing the possibility of fitting axes of species 1--32 among them.
As the characteristic feature of these twills is diagonal lines of, say, dark cells, axes of the fabric to be coloured can be specified as through, between, or perpendicular to the `dark lines' of the redundancy twill.
As Figure~\ref{fig:8dea} indicates, there is no shortage of mirrors and axes of glide-reflections perpendicular to the dark lines; they occur at every opportunity, a mirror and two axes not in mirror position through every cell.
No ingenuity is required to fit axes in this direction, and little in the direction of the dark lines for various clumps of species.

For subspecies $1_o$, $2_o$, and $4_o$, successive glide-reflection axes not in mirror position an odd multiple $m\geq 1$ of $3\beta$ apart can be placed along the glide-reflection axes between the dark lines that have that spacing.
An example of this treatment is Figure~\ref{fig:9ab}, where a species-$1_o$ fabric is 3-coloured this way; the redundant cells run half-way between the fabric's axes of side-preserving glide-reflection, which are not in mirror position.
Since the length of the lattice unit in $\delta$ units is 5 and the number of colours 3, the length of the lattice unit for fixed-colour translational symmetry is 15.
Other ways of arranging the colouring of these and other subspecies ($1_e$, $2_e$, and $4_e$, for example) are possible; all that is being indicated here is that there are {\it some} fabrics of each species 1--32 that can be so coloured.

For species $1_m$, $2_m$, 3, $5_o$, and $7_o$, alternate glide-reflection axes in mirror position ($1_m$, $2_m$, 3) or alternate mirrors ($5_o$, $7_o$) an odd multiple $m\geq 1$ of $3\delta$ apart can be placed along the mirrors in the dark lines that have that spacing.
Then the other axes will fall on intervening mirrors. Glide-reflection axes can lie on these mirrors regardless of parity of glide.
An example of this is Figure~\ref{fig:10ab},
% Could use species 3 2003 12-79-1.
where Roth's \cite{Roth1993} example of subspecies $1_m$ is 3-coloured this way.
\begin{figure}%10, 5 before
\[\begin{array}{cc}
\epsffile{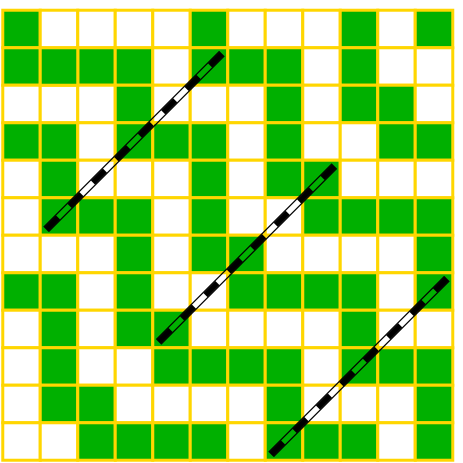} &\epsffile{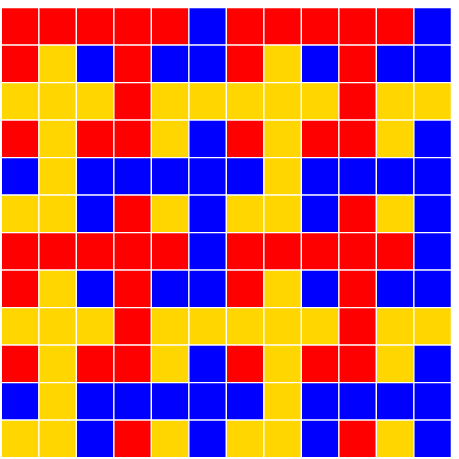}\\%30%[middle p. 22]
\mbox{(a)} &\mbox{(b)}
\end{array}\]
\caption{a. Design of 12-139-1 of species 6. b. Three-colouring by thin striping with redundant cells along the axes.}\label{fig:11ab}
\end{figure}

For species $5_e$ and 6, successive mirrors that can be an odd multiple $m\geq 1$ of $3\delta$ apart can lie along mirrors in the dark lines that have that spacing.
An example is shown thinly striped with 3 colours in Figure~\ref{fig:11ab}.

\begin{figure}%11, 6 before
\[\begin{array}{cc}
\epsffile{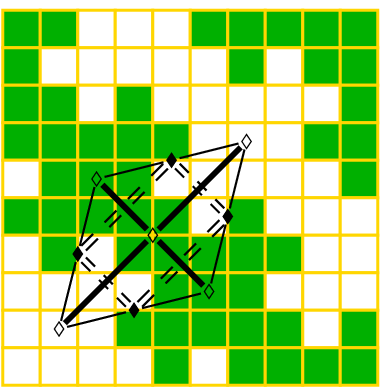} &\epsffile{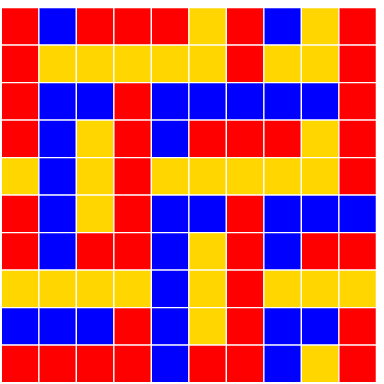}\\%30%[(MS p. 11 middle example of perp 17a)]
\mbox{(a)} &\mbox{(b)}
\end{array}\]
\caption{a. Design of order 30 and species 32 from Figure 17a of \cite{Thomas2010a}. b. Three-colouring by thin striping with redundant cells along the mirrors of positive slope.}\label{fig:12ab}
\end{figure}

For species $8_o$ and 10, $28_o$ and 32, successive mirrors that can be an odd multiple $m\geq 1$ of $3\beta$ apart can be put on dark lines' mirrors.
Intervening glide-reflection axes will lie on intervening axes parallel to the dark lines, and the fabric's perpendicular mirrors and glide-reflection axes, if any ($28_o$ and 32), can be put on the perpendicular mirrors and axes respectively, which are close enough together to make this possible.
An example of species 32 (with the spacing of $28_o$) is shown in Figure~\ref{fig:12ab}.
The disappearance of the mirror symmetry and half of the half-turns makes such a pattern less attractive than one hopes for unless there is extraneous symmetry introduced, and there is hardly room here. 

\begin{figure}%12, 7 before
\[\begin{array}{cc}
\epsffile{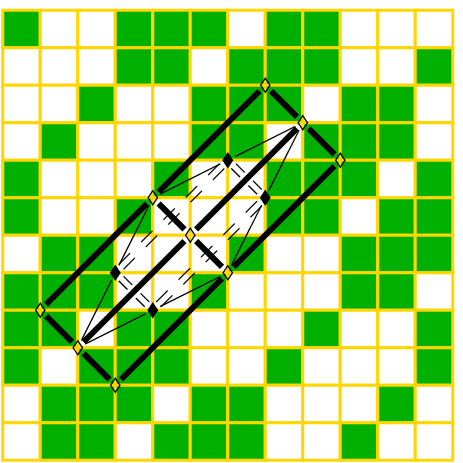} &\epsffile{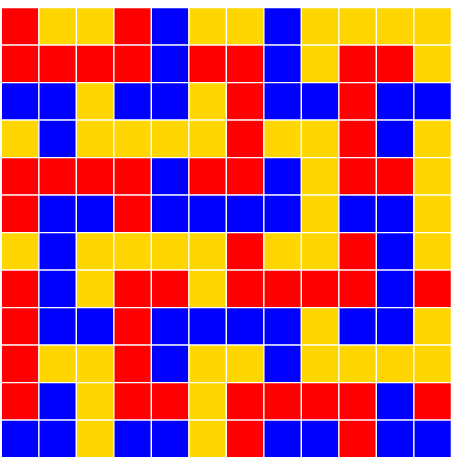}\\%30%[p. 11, top]
\mbox{(a)} &\mbox{(b)}
\end{array}\]
\caption{a. Roth's \cite{Roth1993} example 12-315-4 of species 30. b. Three-colouring by thin striping with redundant cells along the mirrors of negative slope.}\label{fig:13ab}
\end{figure}

For species $8_e$, 9, $27_o$, $28_e$, $28_n$, 30, and 31, successive mirrors that can be an odd multiple $m\geq 1$ of $3\delta$ apart can be put on dark lines' mirrors that are so spaced.
Intervening glide-reflection axes will then lie on intervening parallel mirrors, and then the fabric's perpendicular mirrors and glide-reflection axes ($27_o$, $28_e$, $28_n$, 30, 31) can be put on the perpendicular mirrors and axes respectively.
An example with the same motif as 12-183-1 (Figure~\ref{fig:10ab}b) arranged in a more complex way is 12-315-4 of species 30 in Figure~\ref{fig:13ab}b.
Another example is the species-31 fabric of Figure~\ref{fig:44a}a, which produces a slightly complicated zig-zag when thinly striped with three colours, visible in Figure~\ref{fig:44a}b if one considers the dark and pale versions of the three colours green, red, and blue to be merged to just the three colours green, red, and blue.
This example will be referred to later.
\begin{figure}%13, 8 before
\epsffile{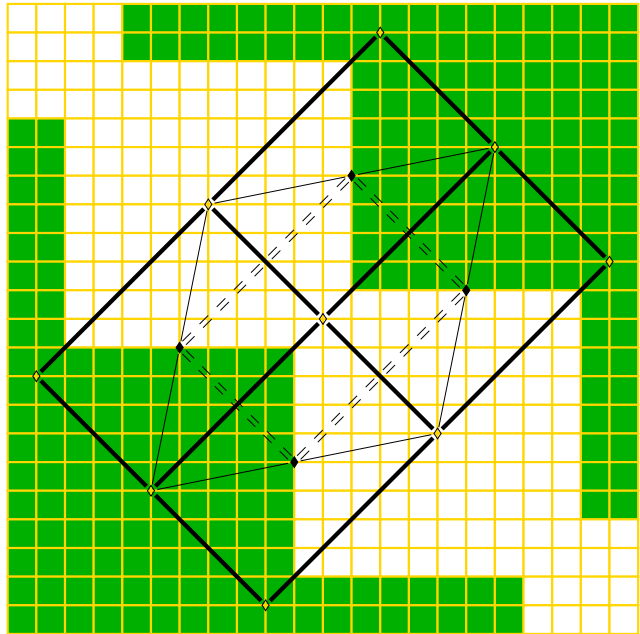}%30%
\caption{Order-48 fabric of species 29; the fabric of Figure~\ref{fig:44a}a doubled.}\label{fig:14a}
\end{figure}
\begin{figure}%14, 9 before
\epsffile{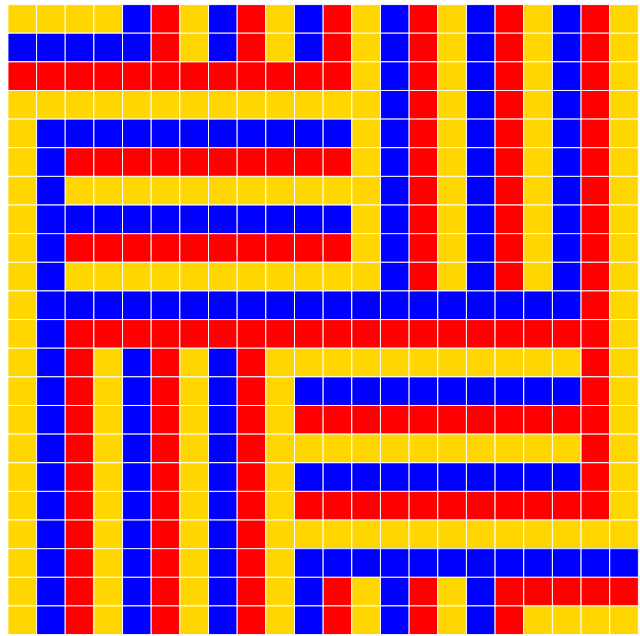}%30%[p. 10 of 29 with $27_e$ spacing]
\caption{Three-colouring of the design of Figure~\ref{fig:14a} by thin striping with redundant cells along the mirrors of negative slope.}\label{fig:14b}
\end{figure}

The previous paragraph can be repeated with $6\delta$ in place of $3\delta$ for species $27_e$ and 29.
An example is shown in Figures~\ref{fig:14a} and \ref{fig:14b}.

\begin{figure}%15, 10 before
\[\begin{array}{cc}
\epsffile{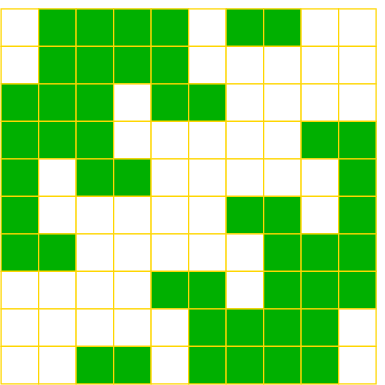} &\epsffile{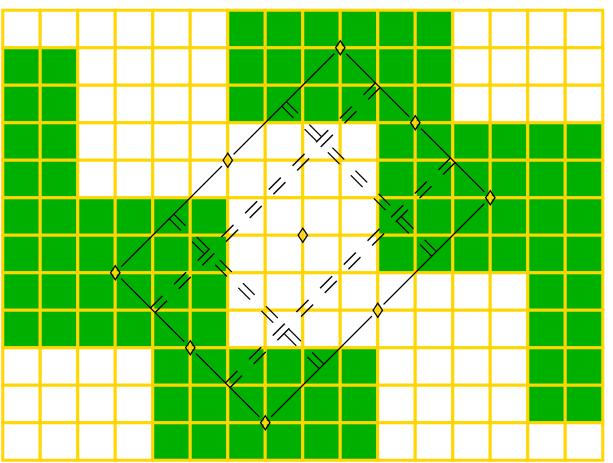}\\%30%[p. 12, p. 8 top]
\mbox{(a)} &\mbox{(b)}
\end{array}\]
\caption{Fabrics coloured in Figure~\ref{fig:17a}. a. Roth's \cite{Roth1993} example 12-111-2 of species $11_o$. b. Example of species $11_e$ from Figure 5a of \cite{Thomas2010a}.}\label{fig:40a}
\end{figure}
\begin{figure}%16, 11 before
\[\begin{array}{cc}
\epsffile{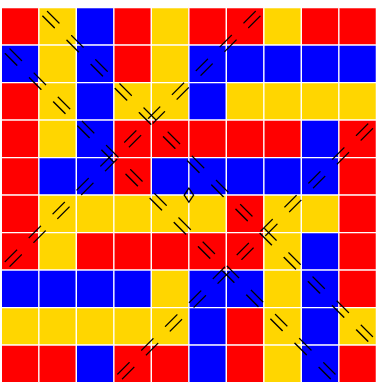} &\epsffile{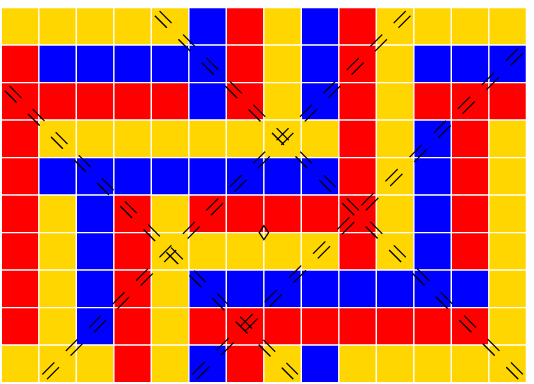}\\%30%[p. 12, p. 8 top]
\mbox{(a)} &\mbox{(b)}
\end{array}\]
\caption{Three-colourings by thin striping with some axes of fabric marked. a. Example of Figure~\ref{fig:40a}a. b. Example of Figure~\ref{fig:40a}b.}\label{fig:17a}
\end{figure}
\begin{figure}%17, 12 before
\epsffile{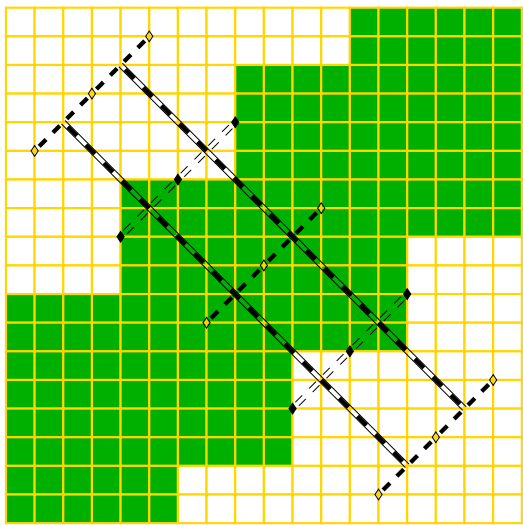}%30%[p. 12]
\caption{Design of a fabric of species 21 and order 96.}\label{fig:18a}
\end{figure}
\begin{figure}%18, 13 before
\noindent
\epsffile{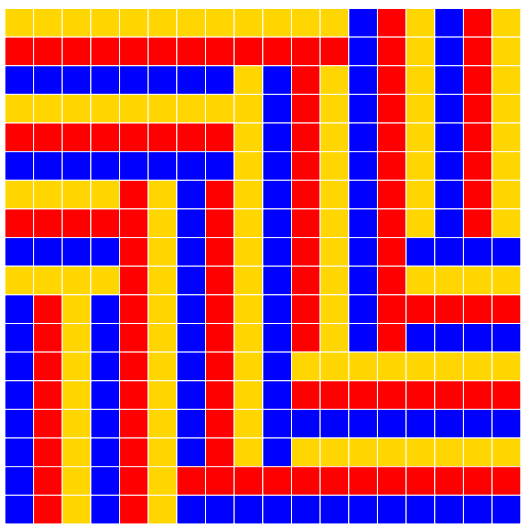}%30%[p. 12]
\caption{Three-colouring by thin striping of the fabric of Figure~\ref{fig:18a} with redundant cells along the axes of positive slope. Much of the symmetry disappears but a half-turn centre remains in the exact centre of the diagram and along a diagonal of positive slope.}\label{fig:18b}
\end{figure}
\begin{figure}%19, 14 before
\[\begin{array}{cc}
\epsffile{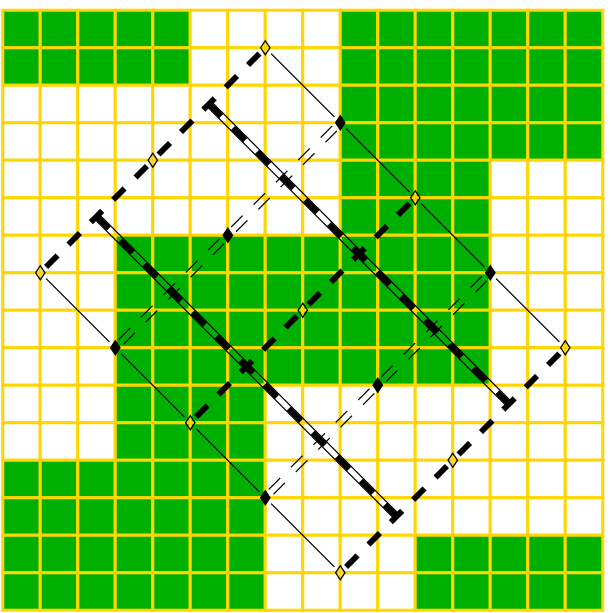} &\epsffile{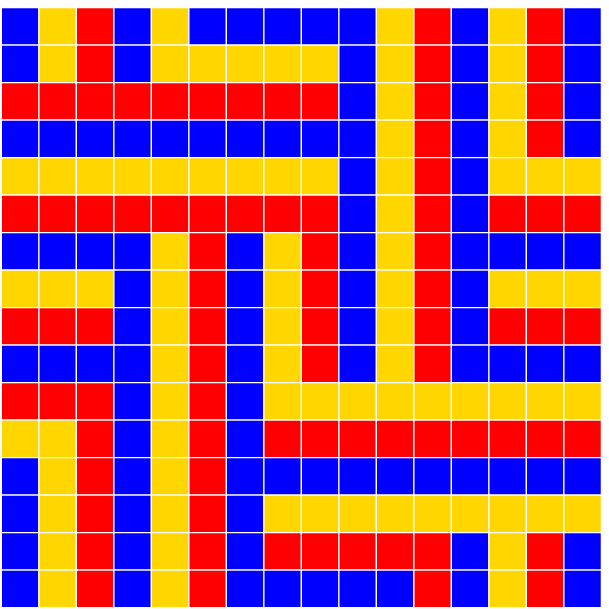}\\%30%[p. 9]
\mbox{(a)} &\mbox{(b)}
\end{array}\]
\caption{a. Species-22 example from Figure 11a of \cite{Thomas2010a}. b. Three-colouring by thin striping with redundant cells along the mirrors of negative slope, which appear as axes of glide-reflection.}\label{fig:20a}
\end{figure}
\begin{figure}%20, 15 before
\[\begin{array}{cc}
\epsffile{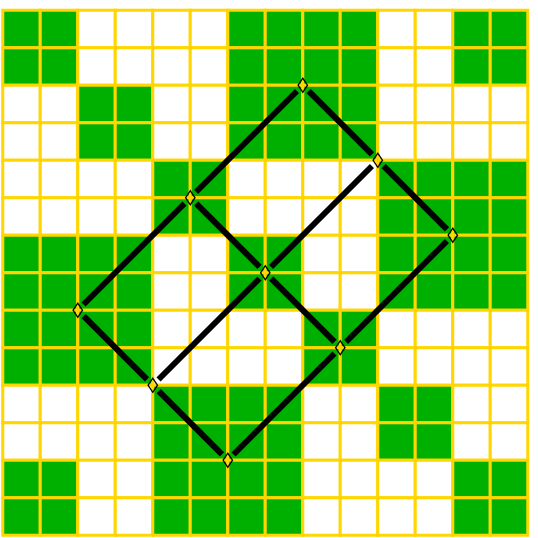} &\epsffile{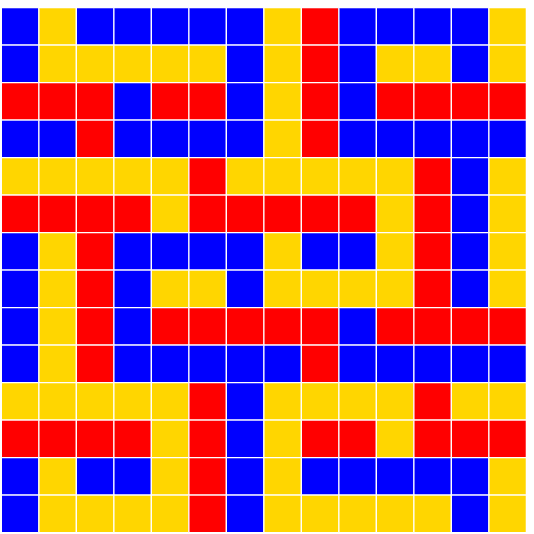}\\%30%[p. 15, top]
\mbox{(a)} &\mbox{(b)}
\end{array}\]
\caption{a. Subspecies-$25_e$ example; the fabric 12-619-1 of Figure~\ref{fig:27a} doubled. b. Three-colouring by thin striping with redundant cells along the mirrors of negative slope.}\label{fig:22a}
\end{figure}
For species 11, 13, 15, 17, 19, 21, 22, and 25, successive axes in mirror position or mirrors that can be an odd multiple $m\geq 1$ of $3\delta$ apart can be put on some dark lines' mirrors.
Then the fabric's perpendicular axes in mirror position and/or mirrors can be put on the mirrors perpendicular to the dark lines.
Examples are Roth's \cite{Roth1993} example of species $11_o$ shown in Figure~\ref{fig:40a}a, the example of species $11_e$ shown in Figure~\ref{fig:40a}b, an example of species 21 in Figure~\ref{fig:18a}, the example of species 22 from Figure 11a of \cite{Thomas2010a} in Figure~\ref{fig:20a}a, and an example of subspecies $25_o$ in Figure~\ref{fig:22a}a, coloured respectively in Figures~\ref{fig:17a}a, \ref{fig:17a}b, \ref{fig:18b}, \ref{fig:20a}b, and \ref{fig:22a}b.

\begin{figure}%21, 16 before
\[\begin{array}{cc}
\epsffile{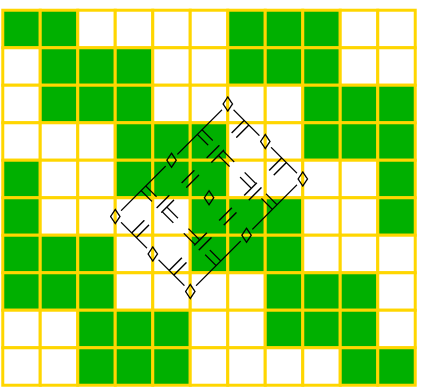} &\epsffile{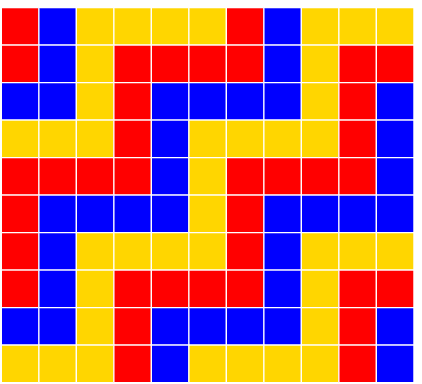}\\%30%[mid p.8]
\mbox{(a)} &\mbox{(b)}
\end{array}\]
\caption{a. 12-231-1 (Fig.~6a of \cite{Thomas2010a}) of subspecies $12_e$. b. Three-colouring by thin striping with redundant cells along the axes in mirror position (negative slope).}\label{fig:23a}
\end{figure}
\begin{figure}%22, 17 before
\[\begin{array}{cc}
\epsffile{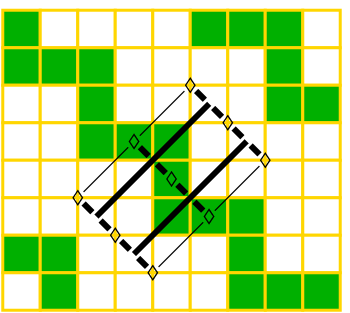} &\epsffile{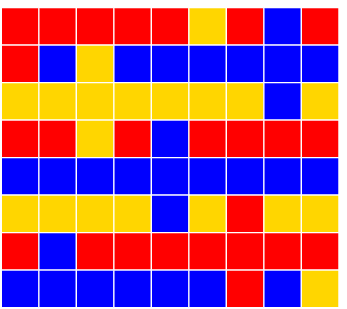}\\%30%[mid p.8]
\mbox{(a)} &\mbox{(b)}
\end{array}\]
\caption{a. 12-135-1 (Fig.~9b of \cite{Thomas2010a}) of subspecies $18_s$. b. Three-colouring by thin striping with redundant cells along the glide-reflection axes.}\label{fig:24a}
\end{figure}
\begin{figure}%23, 18 before
\[\begin{array}{cc}
\epsffile{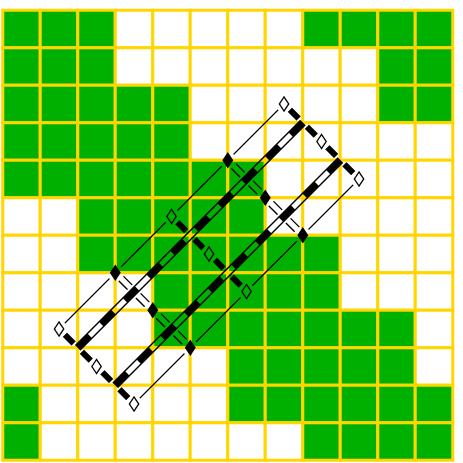} &\epsffile{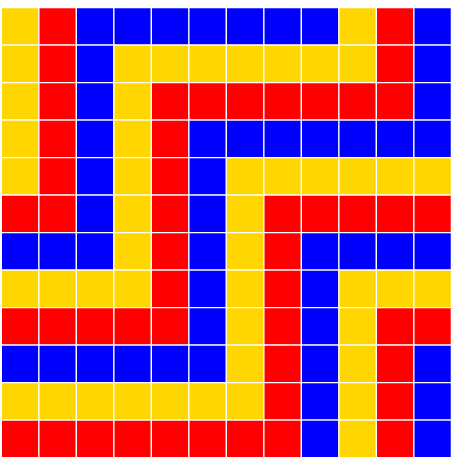}\\%30%[mid p.9]
\mbox{(a)} &\mbox{(b)}
\end{array}\]
\caption{a. 12-31-1 (Fig.~11b of \cite{Thomas2010a}) of subspecies $23_o$. b. Three-colouring by thin striping with redundant cells along the glide-reflection axes of negative slope.}
\label{fig:25a}
\end{figure}
\begin{figure}%24, 19 before
\[\begin{array}{cc}
\epsffile{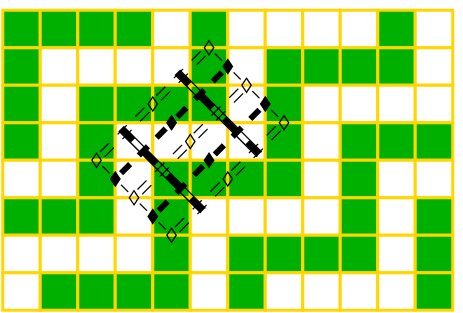} &\epsffile{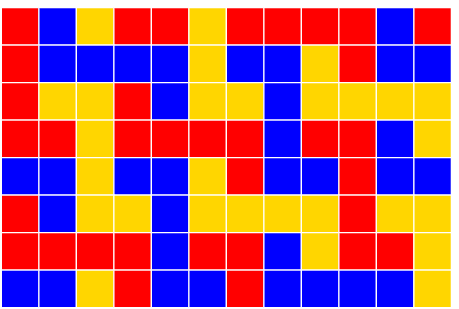}\\%30%[p.9 bottom]
\mbox{(a)} &\mbox{(b)}
\end{array}\]
\caption{a. 12-189-1 of subspecies $24_e$. b. Three-colouring by thin striping with redundant cells along the mirrors.}
\label{fig:26a}
\end{figure}
\begin{figure}%25, 20 before
\[\begin{array}{cc}
\epsffile{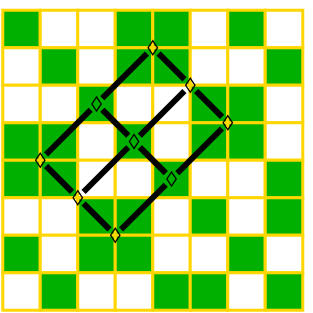} &\epsffile{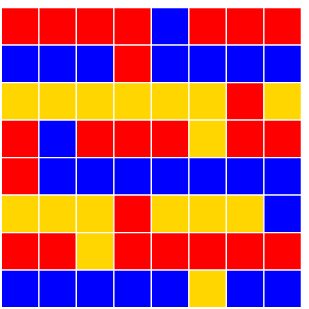}\\%30%[p.7 bottom]
\mbox{(a)} &\mbox{(b)}
\end{array}\]
\caption{a. 12-619-1 of subspecies $26_e$ (Fig.~4b of \cite{Thomas2010a}). b. Three-colouring by thin striping with redundant cells along the mirrors of negative slope.}\label{fig:27a}
\end{figure}

For species $12_e$, $14_e$, 16, 18, 20, 23, 24, and 26, successive glide-reflection axes in mirror position that can be an odd multiple $m\geq 1$ of $3\beta$ apart can be put on the dark lines' mirrors.
The fabric's perpendicular axes not in mirror position and perpendicular mirrors can be put on the axes and mirrors perpendicular to the dark lines.
Examples are Roth's \cite{Roth1993} example of species 12 in Figure~\ref{fig:23a}a and coloured in Figure~\ref{fig:23a}b (having mirrors in the pattern along what are axes in the design); Roth's \cite{Roth1993} example of species 18 in Figure~\ref{fig:24a}a and coloured in Figure~\ref{fig:24a}b (symmetry except for half-turns, of which there are a lot, and translations disappears); an example of a subspecies-$23_o$ fabric in Figure~\ref{fig:25a}a and coloured in Figure~\ref{fig:25a}b (symmetry gained as the side-preserving glide-reflections become mirrors too at the same time as the mirrors become visible glide-reflections); Roth's \cite{Roth1993} example of species 24 in Figure~\ref{fig:26a}; and Roth's \cite{Roth1993} example of species 26 in Figure~\ref{fig:27a}.

In the above paragraphs, the dimensions $3\beta$, $3\delta$, and $6\delta$ for three colours need only be changed to $p\beta$, $p\delta$, and $2p\delta$ for any odd number of colours $p$ to specify how thin striping can be done in each possible species.
\begin{Thm}%4
In each of the species $1$--$32$, that is, all those allowed by the quarter-turn ban, there are fabrics that can be perfectly coloured by thin striping with an odd number $2q+1$ of colours and redundant cells arranged as a $2q/1$ twill for $q=1, 2, \dots$.
\label{thm:5}
\end{Thm}

\section{Six colours}%4
\label{sect:6C}

\noindent When the number of colours is $4p$, the redundant twill is of the form $(4p-1)/1$ of subspecies $26_e$ (Figure~\ref{fig:28b}a).
When the number of colours is congruent to 2 mod 4, which is what we pursue in this section, the redundant twill is of the form $(4p+1)/1$ of subspecies $26_o$ (Figure~\ref{fig:28b}b).
In the former case the mirrors between the dark lines have a cell-centre \dia half-way between cell-centre \diaas in redundant cells and a cell-corner \dia half-way between cell-corner \diaas where redundant cells touch.
In the latter case, the cell-centre \dia in the middle lies between cell-corner \diaas in the dark lines and vice versa.
In both cases, there are only mirrors every $\beta$ perpendicular to the dark lines.
Therefore for all even numbers of colours no glide-reflection axis not in mirror position can be accommodated.
This rules out subspecies $1_e$, $1_o$, $2_e$, $2_o$, $8_o$, $18_o$, $18_e$, and $28_o$ and whole species 4, 10, 12, 14, 16, 20, 24, and 32 as for two colours. 
Call this the mirror-position requirement.
Subspecies left from partly affected species are $1_m$, $2_m$, $8_e$, $18_s$, $28_e$, and $28_n$.
The quarter-turn ban eliminates species 33--39.
Left are  $1_m$, $2_m$, 3, $5$--$7$, $8_e$, 9, $11$, $13$, $15$, $17$, $18_s$, $19$, 21--23, $25$, $26$, $27_e$ and 29,  $27_o$ and 30, $28_e$ and 31, and $28_n$.

\begin{figure}%26, 21 before
\[\begin{array}{cc}
\epsffile{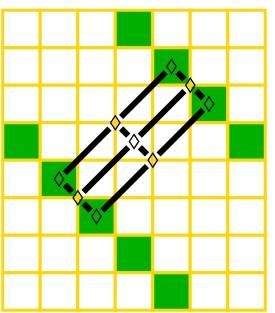} &\epsffile{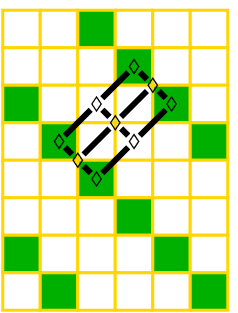}\\%30%
\mbox{(a)} &\mbox{(b)}
\end{array}\]
\caption{Twills for colouring with an even number of colours, redundant cells for thin striping dark. a. 5/1 twill for six-colouring. b. 3/1 twill for four-colouring.}\label{fig:28b}
\end{figure}

Because the redundant twill 5/1 for thin striping with six colours is of subspecies $26_o$, no further subspecies are ruled out than those just ruled out for all even numbers of colour.
Since the aim here is to show which species contain {\it some} perfectly colorable designs, $28_n$ need not be separately considered.
All of those listed do contain 6-colorable designs, i.e., all but those eliminated by the mirror-position requirement and the quarter-turn ban.
Since the symmetry group for the 5/1 twill is a subgroup of the symmetry group of both the 2/1 twill and plain weave, 6-colorability, which I shall use as a shorthand for perfect colorability by thin striping with 6 colours, implies both 3-colorability and 2-colorability by thin striping.
Three-colorability of a fabric does not ensure 6-colorability because of the mirror-position requirement.
Two-colorability of a fabric does not ensure 6-colorability because the symmetry group of plain weave is very much larger (both infinite of course) than that of the 5/1 twill, but all of the species that contain 2-colorable fabrics do also contain 6-colorable fabrics.
\begin{figure}%27, 22 before
\[\begin{array}{cc}
\epsffile{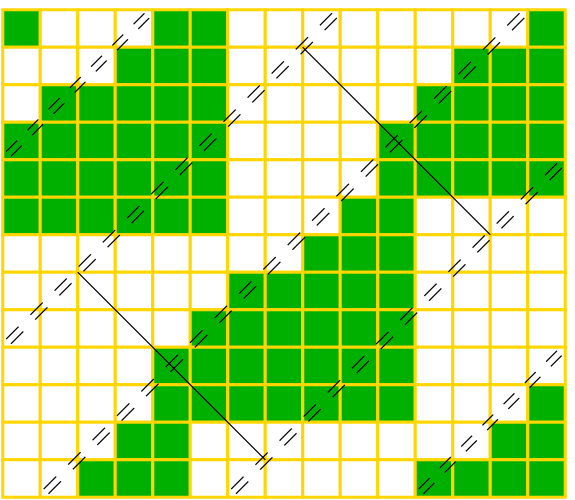} &\epsffile{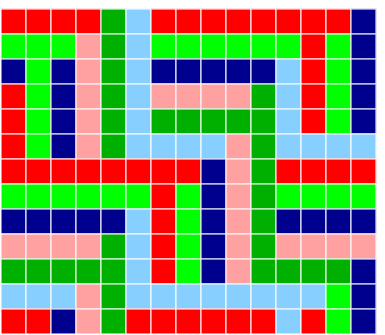}\\%30%&20% [p.8]
\mbox{(a)} &\mbox{(b)}
\end{array}\]

\caption{a. Design of fabric of species $1_m$ with glide 3.\hskip 15 pt b. Six-colouring of fabric of Figure~\ref{fig:30}a with lines of redundant cells across the axes of symmetry.}\label{fig:30}
\end{figure}

\begin{figure}%28, not 23
\[\begin{array}{cc}
\epsffile{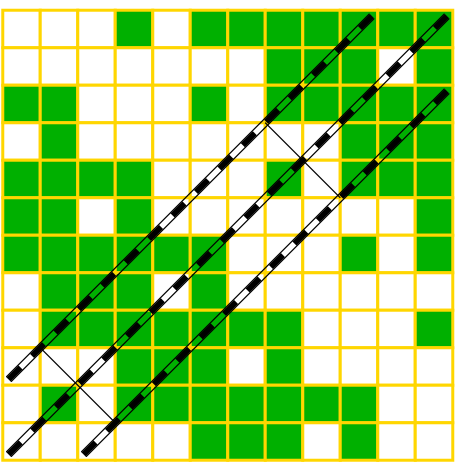} &\epsffile{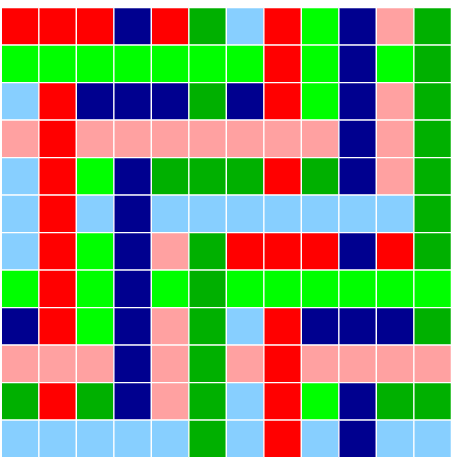}\\%30%new sheet
\mbox{(a)} &\mbox{(b)}
\end{array}\]
\caption{a. Design of 12-23-1, Roth's \cite{Roth1993} example of species 6 (Fig.~6b of \cite{Thomas2009}). b. Six-colouring by thin striping with lines of redundant cells across axes.}\label{fig:32a}
\end{figure}

When glide-reflection axes are in mirror position as they must be here, the length of a lattice unit, where there are only parallel axes, must be even (twice the glide).
For species $1_m$, $2_m$, 3, 6, and $7$ (where glide-reflection axes are also mirrors), axes of glide-reflection with glide a multiple $m\geq 1$ of $3\delta$ can be placed on mirrors perpendicular to the dark lines.
For example, a fabric of species $1_m$ is illustrated in Figure~\ref{fig:30} and a Roth \cite{Roth1993} example of species 6 is illustrated in Figure~\ref{fig:32a}.
For subspecies $5_e$, provided that the length of the lattice unit is a multiple $m\geq 1$ of $3\delta$ (there is no glide-reflection), the mirrors can also be placed on mirrors perpendicular to the dark lines.
Because the only symmetry axes are mirrors, which involve reflection in the plane of the fabric ($\tau$), the only symmetry visible on a single side of such a fabric---unless an extraneous symmetry is introduced---is translational.
To check the perfection of the colouring it is necessary to see the reverse.
Despite the involvement of $\tau$, the reverse is not the same coloured pattern as the obverse reflected because it is coloured differently.
We shall however continue the policy of not illustrating reverses.
Species $8_e$ and 9 with glides an odd multiple $m\geq 1$ of $3\delta$ can be treated like the previous subspecies.
An example from species 9 (so as to have side-preserving glide-reflections) is illustrated in Figure~\ref{fig:33}.
These arrangements have been chosen for simplicity of exposition; there are other possibilities.

\begin{figure}%29 not 24
\[\begin{array}{cc}
\epsffile{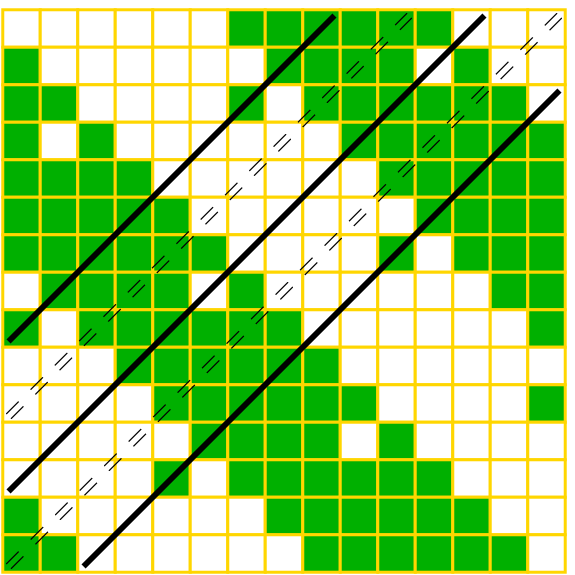} &\epsffile{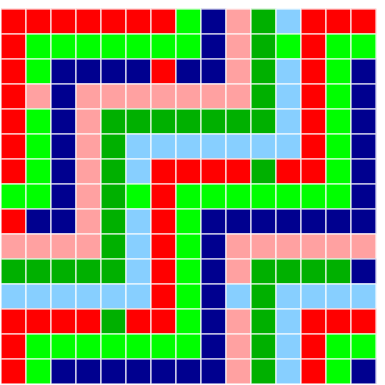}\\%30%&20%p. 9
\mbox{(a)} &\mbox{(b)}
\end{array}\]

\caption{a. Design of fabric of species 9. \hskip 15 pt b. Six-colouring of fabric of Figure~\ref{fig:33}a with lines of redundant cells across axes.}\label{fig:33}
\end{figure}
\begin{figure}%30 not 25
\epsffile{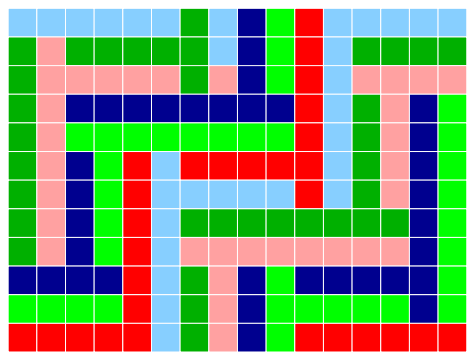}%30%
\caption{Six-colouring of fabric of Figure~\ref{fig:40a}b with lines of redundant cells along axes with negative slope.}
\label{fig:40b}
\end{figure}
\begin{figure}%31 not 26
\epsffile{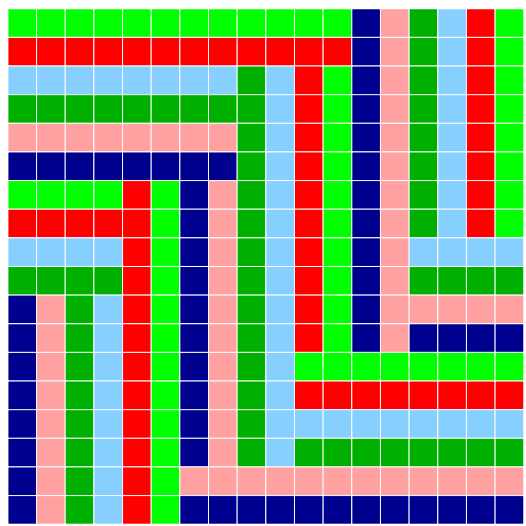}%30%top rt p. 4 from p.12 top of 3-col.
\caption{Six-colouring of the species-21 design of Figure~\ref{fig:18a} with lines of redundant cells along axes with positive slope.}
\label{fig:41(6)}
\end{figure}%32 not 27
\begin{figure}
\epsffile{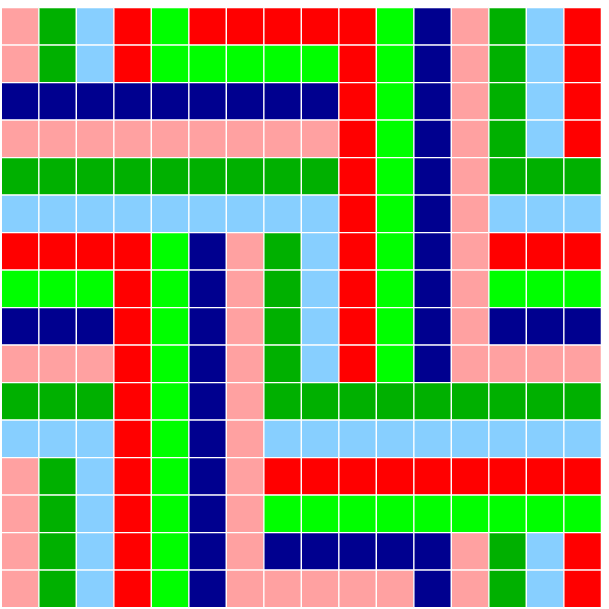}%30%perp 11a p.4
\caption{Six-colouring of the species-22 design of Figure~\ref{fig:20a}a with lines of redundant cells along mirrors/axes with negative slope.}\label{fig:42}
\end{figure}
\begin{figure}%33 not 28
\[\begin{array}{cc}
\epsffile{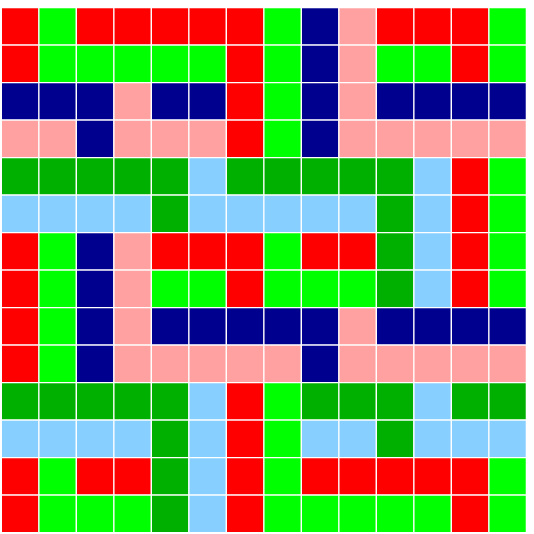} &\epsffile{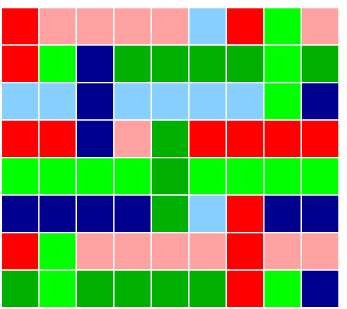}\\%30%p. 14 of 3-col. in middle rt p. 4 (43 from 21a --- file 22a) & perp 9b
\mbox{(a)} &\mbox{(b)}
\end{array}\]

\caption{a. Six-colouring of the species-25${}_e$ design of Figure~\ref{fig:22a}a with lines of redundant cells along mirrors with negative slope.\hskip 15 pt b. Six-colouring of species-$18_s$ fabric of Figure~\ref{fig:24a}a with redundant cells along axes of negative slope between the zig-zags.}\label{fig:43}
\end{figure}

For species 11, 13, 15, 17,, 19, 21, 22, and 25, successive axes or mirrors that can be an odd multiple $m\geq 1$ of $3\delta$ apart can be placed on dark lines' mirrors as in 3-colouring.
Then the fabric's perpendicular axes, which are in mirror position (the only constraint on them), can be put on mirrors perpendicular to the dark lines.
Examples are illustrated for species $11_e$ (Figg.~\ref{fig:40a}b and \ref{fig:40b}), 21 (Figg.~\ref{fig:18a} and \ref{fig:41(6)}), 22 (Figg.~\ref{fig:20a}a and \ref{fig:42}), and 25 (Figg.~\ref{fig:22a}a and \ref{fig:43}a).

\begin{figure}%34 not 29
\[\begin{array}{cc}
\epsffile{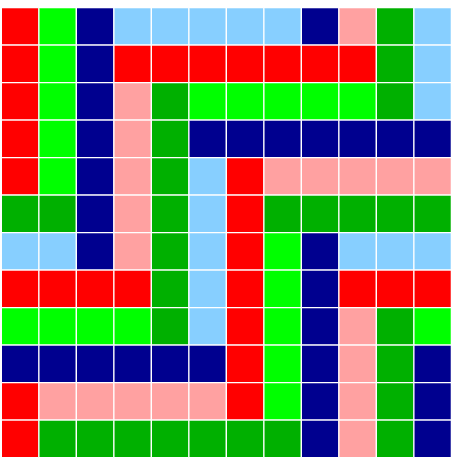} &\epsffile{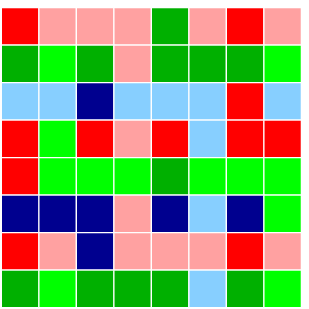}\\%30%perp  11b & perp 4b for 26e p.5.
\mbox{(a)} &\mbox{(b)}
\end{array}\]
\caption{a. Six-colouring by thin striping of species-$23_o$ fabric of Figure~\ref{fig:25a}a with lines of redundant cells along the axes of side-preserving glide-reflection that are not also mirrors.\hskip 15 pt b. Six-colouring of the species-$26_e$ design of Figure~\ref{fig:27a}a with lines of redundant cells along mirrors with negative slope.}\label{fig:38b}
\end{figure}

For species $18_s$, $23_o$, and $26_e$ with cell-centre \diaa, successive axes, axes, and mirrors respectively that can be an odd multiple $m\geq 1$ of $3\beta$ apart can be placed on the mirrors along and between the dark lines.
The fabric's mirrors perpendicular to these will fall on mirrors perpendicular to the dark lines.
Examples of 6-colourings are of the species-$18_s$ fabric of Figure~\ref{fig:24a}a in Figure~\ref{fig:43}b, the species-$23_o$ fabric of Figure~\ref{fig:25a}a in Figure~\ref{fig:38b}a, and the species-$26_e$ fabric of Figure~\ref{fig:27a}a, in Figure~\ref{fig:38b}b, where only the half-turns and translational symmetry remain in the pattern.

For species 27--31, central rectangles of $27_o$ and 30, $28_e$ and 31 (and $28_n$), can have one dimension an odd multiple $m\geq 1$ of $3\delta$ and $27_e$ and 29 can have one dimension an odd multiple $m\geq 1$ of $6\delta$.
The former can have these mirrors placed on dark lines' mirrors, axes falling between them, and the latter can have mirrors and axes placed on dark lines' mirrors.
Fabrics' perpendicular mirrors and axes fall on mirrors and axes perpendicular to the dark lines.
An example is a species-31 fabric (Figure~\ref{fig:44a}a) 6-coloured with redundant cells along the mirrors of negative slope in Figure~\ref{fig:44a}b.
The mirror symmetry and side-reversing half-turns are missing from the pattern, but the ordinary half-turns are there in the middle of what can be seen as tri-colour flags and the centre of what can be seen as flag poles.
(Each flag and pole has the other motif at both ends---in the same colour or colours---rather than just the normal one on account of the half-turn symmetries.)
\begin{figure}%35 not 30
\[\begin{array}{cc}
\epsffile{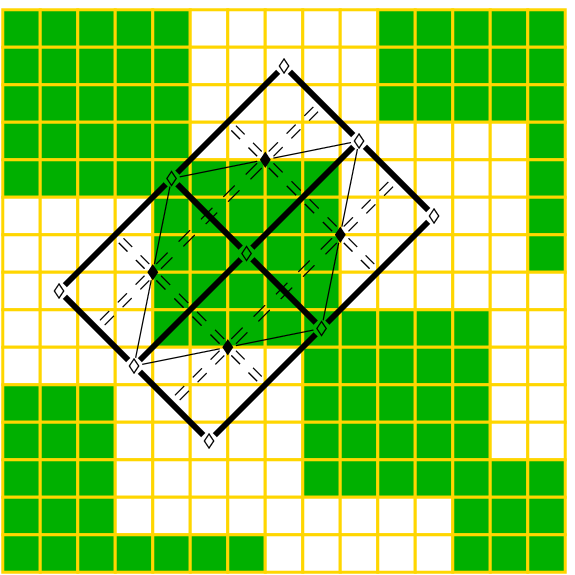} &\epsffile{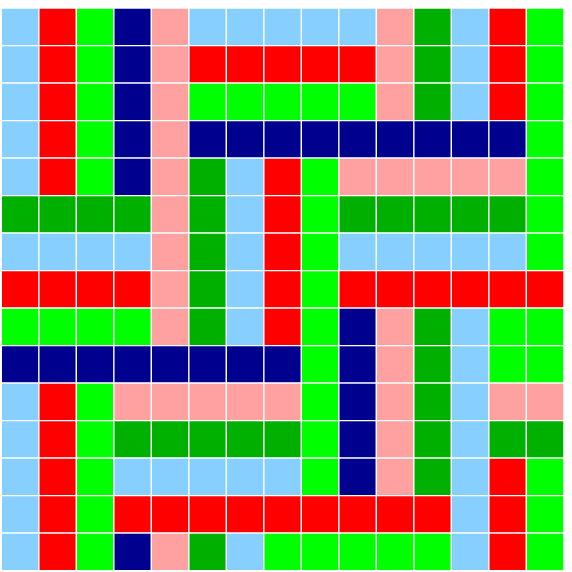}\\%30%perp 15b for 31 on p. 3.
\mbox{(a)} &\mbox{(b)}
\end{array}\]
\caption{a. Species-31 fabric (Fig.~15b of \cite{Thomas2010a}). b. Six-colouring by thin striping with lines of redundant cells along mirrors of negative slope.}\label{fig:44a}
\end{figure}

\begin{Thm}
In each of the species  $1$--$3$, $5$--$9$, $11$, $13$, $15$, $17$--$19$, $21$--$23$, and $25$--$31$, that is, all those allowed by the mirror-position requirement on glide-reflection axes and the quarter-turn ban, there are fabrics that can be perfectly coloured by thin striping with $4p+2$ colours and redundant cells arranged as a $(4p+1)/1$ twill for $p=1, 2, \dots$.
\label{thm:6}
\end{Thm}

\section{Four colours}%5
\label{sect:4C}

\noindent For four colours there is a further constraint because the species-$26_e$ lattice units of the redundant 3/1 twill (Fig.~\ref{fig:28b}b) are odd by even in $\delta$.
Species $27_o$ and 30 (with $27_o$ spacing) have glide-reflection axes an odd distance apart in $\delta$ in both perpendicular directions; they can therefore not be accommodated on mirrors of the redundant twill because the glides in both perpendicular directions must be odd.
In these two species, however, $27_e$ remains; only 30 is eliminated altogether.
However, $1_m$, $2_m$, 3, $5$--$7$, $8_e$, 9, $11$, $13$, $15$, $17$, $18_s$, $19$, $21$--$23$, $25$, $26$, $27_e$, $28$, 29, and 31 can be perfectly coloured by thin striping with four colours, that is, fabrics in all the species but 30 and those eliminated by the mirror-position requirement and the quarter-turn ban.
Since the redundant cells and symmetry axes for four colours are related as they are to those for two colours, the symmetry group of the redundant twill is a subgroup of that for two colours.
Accordingly, if a fabric is four-colorable by thin striping then it is two-colorable by thin striping but not the converse.
The biggest actual difference is the two-colorability of species 30 and $36_s$.
\begin{figure}%36 not 31
\[\begin{array}{cc}
\epsffile{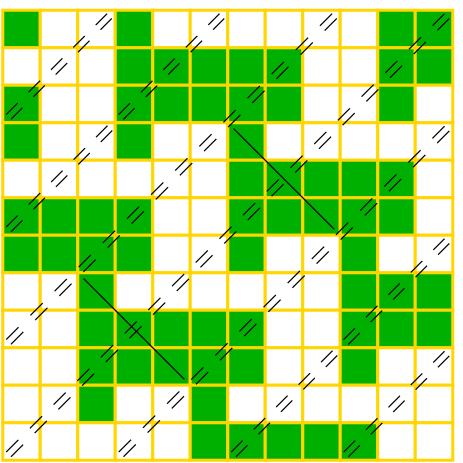} &\epsffile{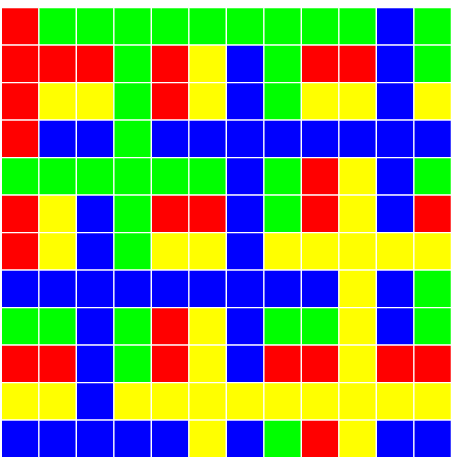}\\%30%new page replacing pp.2,3
\mbox{(a)} &\mbox{(b)}
\end{array}\]
\caption{a. Species-$1_m$ fabric. b. Four-colouring by thin striping with lines of redundant cells perpendicular to the axes.}\label{fig:45a}
\end{figure}

When glide-reflection axes are in mirror position as they must be here, the length of a lattice unit, where there are only parallel axes, must be even (twice the glide); as a result, the width must be odd \cite{Thomas2009} for species $1_m$, $2_m$, 3, $5_e$, and $7_e$.
For species $1_m$, $2_m$, 3, and $7_e$, axes of glide-reflection with glide a multiple $m\geq 1$ of $2\delta$ must therefore be placed on mirrors perpendicular to the dark lines.
For example, Figure~\ref{fig:45a}a shows a species-$1_m$ fabric with glide $2\delta$, and Figure~\ref{fig:45a}b shows it perfectly 4-coloured by thin striping with redundant cells perpendicular to the axes.
For subspecies $5_e$, provided that the length of the lattice unit is a multiple $m\geq 1$ of $2\delta$, the mirrors can be placed on mirrors perpendicular to the dark lines exactly like the glide-reflection axes two sentences back.
An example is Roth's \cite{Roth1993} example of subspecies $5_e$, illustrated in Figure~\ref{fig:46a}a, 4-coloured in Figure~\ref{fig:46a}b with mirrors across the dark lines as they must be.
\begin{figure}%37 not 32
\[\begin{array}{cc}
\epsffile{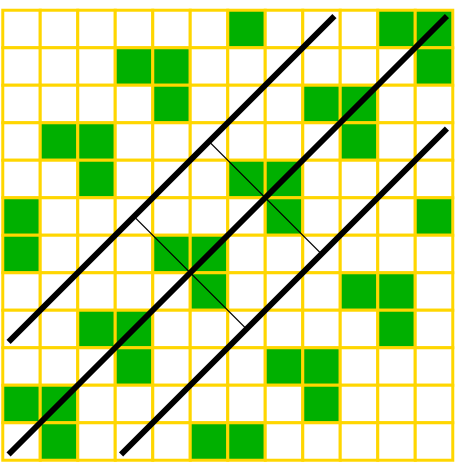} &\epsffile{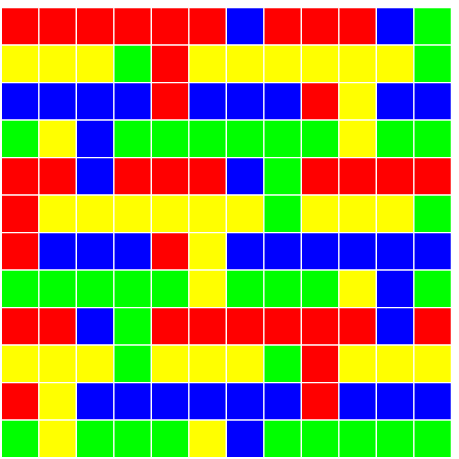}\\%30%(parallel 2a with back on p. 13 bottom)
\mbox{(a)} &\mbox{(b)}
\end{array}\]
\caption{a. The species-$5_e$ fabric 12-35-1 (Fig.~2a of \cite{Thomas2009}). b. Four-colouring by thin striping with lines of redundant cells perpendicular to the axes.}\label{fig:46a}
\end{figure}
Because the only symmetry axes are mirrors, which involve $\tau$, the only symmetry visible on a single side of such a fabric---unless an extraneous symmetry is introduced---is translational.
To check the perfection of the colouring it is necessary to see the reverse.
Despite the involvement of $\tau$, the reverse is not the same coloured pattern as the obverse reflected because it is coloured differently.
We shall however continue the policy of not illustrating reverses.
\begin{figure}%38 not 33
\[\begin{array}{cc}
\epsffile{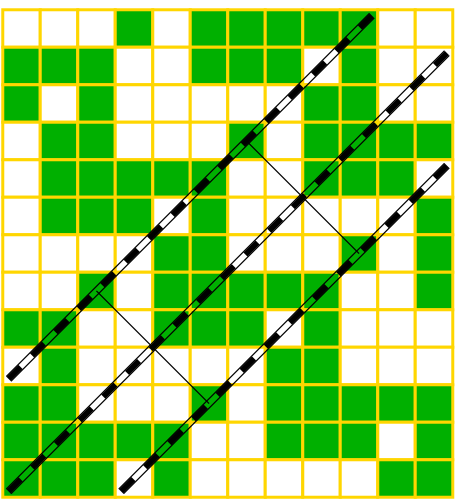} &\epsffile{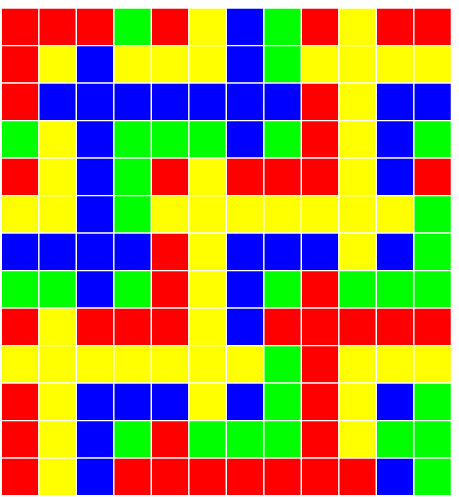}\\%30%parallel 16a p.3
\mbox{(a)} &\mbox{(b)}
\end{array}\]
\caption{a. A species-$7_e$ fabric (Fig.~16a of \cite{Thomas2009}). b. Four-colouring by thin striping with lines of redundant cells perpendicular to the axes.}
\label{fig:47a}
\end{figure}
An example of subspecies $7_e$ is illustrated in Figure~\ref{fig:47a}, where the reflective symmetry with mirrors coincident with the axes of the visible glide-reflections is not visible, but extraneous reflective symmetry introduced by the colouring has mirrors perpendicular to the axes through the lines of redundant cells $2\delta$ apart.
This is side-preserving reflective symmetry, which cannot happen in a design.

\begin{figure}%39 not 34
\[\begin{array}{ccc}
\epsffile{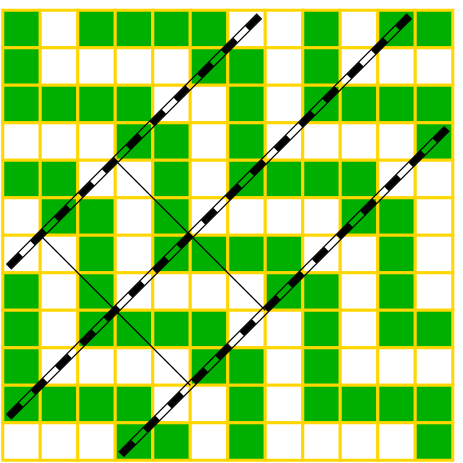} &\epsffile{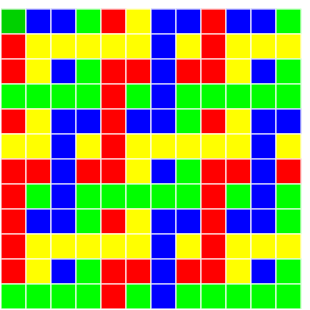} &\epsffile{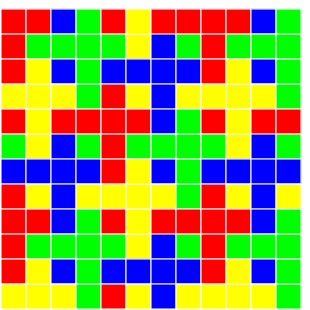}\\%30%parallel 6a p.4
\mbox{(a)} &\mbox{(b)} &\mbox{(c)}
\end{array}\]
\caption{a. 8-11-1 of species 6 (Fig.~6a of \cite{Thomas2009}).\hskip 15 pt b. Four-colouring by thin striping with lines of redundant cells along the mirrors. c. Second four-colouring with lines of redundant cells between the mirrors. Extraneous mirrors are introduced along the lines of redundant cells.}
\label{fig:48a}
\end{figure}

Species-6 lattice units have width a multiple $m\geq 1$ of $2\delta$ and so can have their bounding mirrors placed along or between dark lines' mirrors, their intervening mirrors falling on mirrors between or on dark lines. 
They cannot be arranged perpendicular to the design's axes because the glides are odd.
An example is Roth's \cite{Roth1993} example of species 6 in Figure~\ref{fig:48a}a, coloured two different ways in Figures~\ref{fig:48a}b and c.

\begin{figure}%40 not 35
\[\begin{array}{cc}
\epsffile{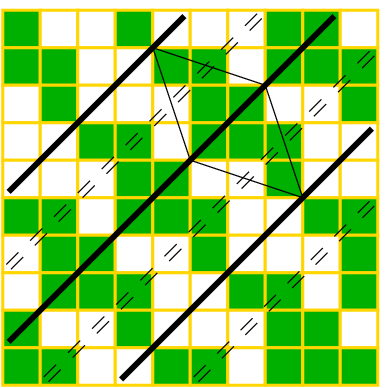} &\epsffile{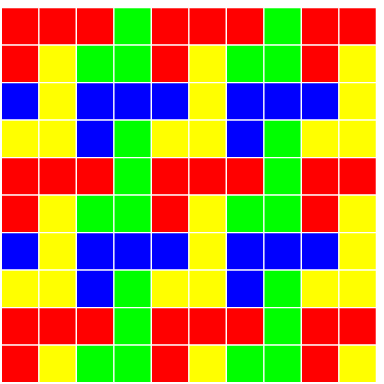}\\%30%parallel 7b p. 4
\mbox{(a)} &\mbox{(b)}
\end{array}\]
\caption{a. Species-$9$ fabric 8-19-2 (Fig.~7b of \cite{Thomas2009}). b. Four-colouring by thin striping, which looks thick, with lines of redundant cells along the mirrors. Extraneous side-preserving glide-reflections are along the mirrors.}\label{fig:49a}
\end{figure}
For species $8_e$ and 9, successive mirrors a multiple $m\geq 1$ of $2\delta$ apart can be placed along the dark lines' mirrors or all between them if the glides are even.
The axes of glide-reflection will then fall on mirrors in or between the dark lines.
Odd glides too can be accommodated by making sure that the mirrors fall all on dark lines or all between them, since that makes room for the axes on or between the lines.
An example with odd glides and the redundant cells along the mirrors is Roth's \cite{Roth1993} example of species 9 (for the sake of having side-preserving glide-reflections) in Figure~\ref{fig:49a}.

\begin{figure}%41 not 36
\[\begin{array}{cc}
\epsffile{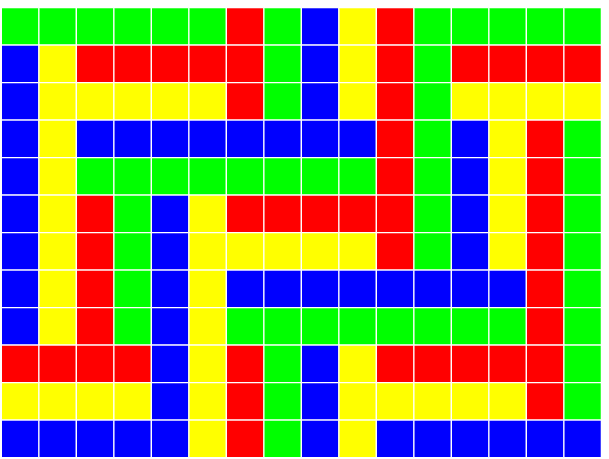} &\epsffile{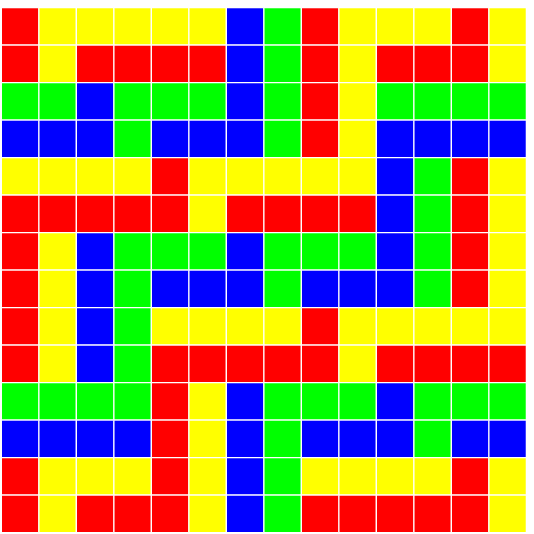}\\%30%a from Fig33=perp5a p.5, b 54 perp 11a for 22 (Fig 20a)
\mbox{(a)} &\mbox{(b)}
\end{array}\]
\caption{a. Species-$11_e$ fabric of Figure~\ref{fig:40a}b (Fig.~5a of \cite{Thomas2010a}) 4-coloured by thin striping with redundant cells along the axes of positive slope. b. Four-colouring of fabric of Figure~\ref{fig:22a}a of species $25_e$ with redundant cells along mirrors of positive slope.}\label{fig:50a}
\end{figure}
\begin{figure}%42 not 37
\[\begin{array}{cc}
\epsffile{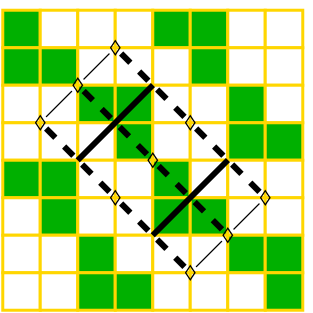} &\epsffile{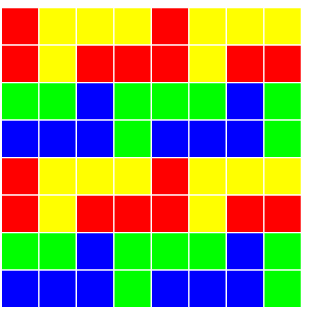}\\%30% perp 9a II for 17e p.5 50d has red.cells along alternate axes:I.
\mbox{(a)} &\mbox{(b)}
\end{array}\]
\caption{a. 8-19-7 of species $17_e$. b. Four-colouring by thin striping with redundant cells along mirrors.}\label{fig:50b}
\end{figure}
\begin{figure}%43 not 38
\[\begin{array}{cc}
\epsffile{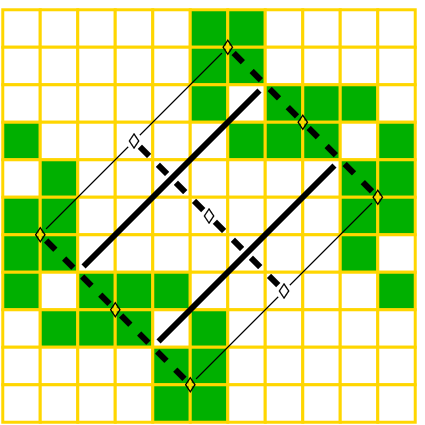} &\epsffile{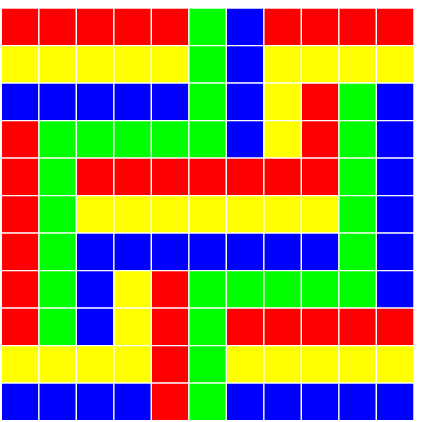}\\%30%made-up ex pp. 5,6
\mbox{(a)} &\mbox{(b)}
\end{array}\]
\caption{a. A fabric of species $18_s$. b. Four-colouring by thin striping with redundant cells along mirrors.}\label{fig:51a}
\end{figure}
\begin{figure}%44 not 39
\[\begin{array}{cc}
\epsffile{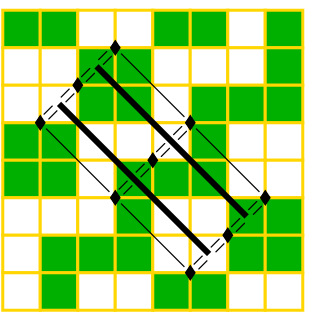} &\epsffile{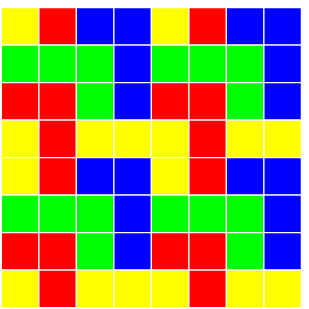}\\%30%perp 9c p.6
\mbox{(a)} &\mbox{(b)}
\end{array}\]
\caption{a. 8-19-4 of species $19_e$ (Fig.~9c of \cite{Thomas2010a}). b. Four-colouring by thin striping with redundant cells along axes.}\label{fig:52a}
\end{figure}
\begin{figure}%45 not 40
\[\begin{array}{cc}
\epsffile{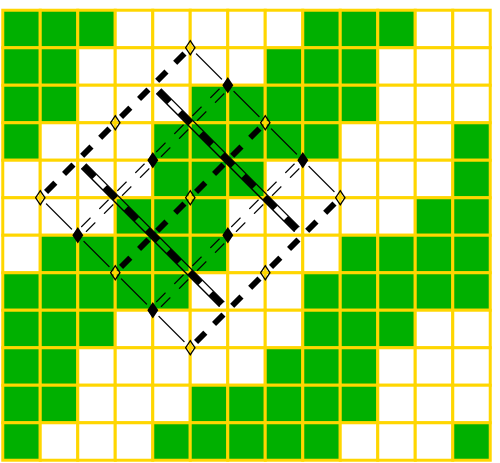} &\epsffile{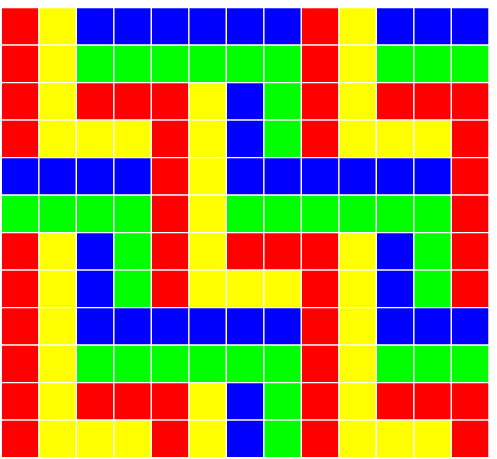}\\%30%Roth's example 8-7-2 perp 10a for 21
\mbox{(a)} &\mbox{(b)}
\end{array}\]
\caption{a. 8-7-2 of species $21$ (Fig.~10a of \cite{Thomas2010a}). b. Four-colouring by thin striping with redundant cells along mirrors.}\label{fig:53a}
\end{figure}
\begin{figure}%46 not 42
\[\begin{array}{cc}
\epsffile{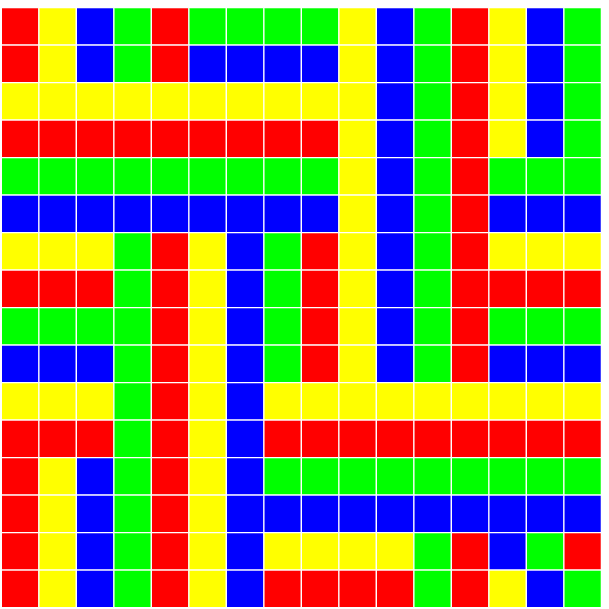} &\epsffile{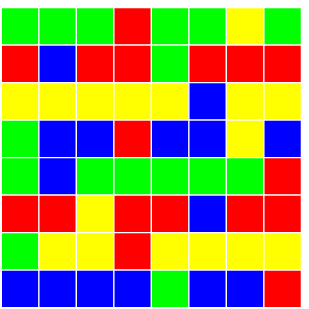}\\%30%made-up 3-col. p. 14 (Fig 21a) on p.7 for 25e, b. perp 4b (Fig 26a) for 26e
\mbox{(a)} &\mbox{(b)}
\end{array}\]
\caption{a. Species-22 fabric of Figure~\ref{fig:20a}a (Fig.~11a of \cite{Thomas2010a}) 4-coloured by thin striping with redundant cells along the axes that are not mirrors (positive slope). b. Four-colouring of fabric of Figure~\ref{fig:27a}a of species $26_e$ with redundant cells along mirrors of positive slope.}\label{fig:54}
\end{figure}
\begin{figure}%47 not 41
\[\begin{array}{cc}
\epsffile{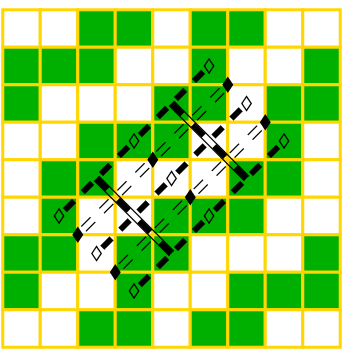} &\epsffile{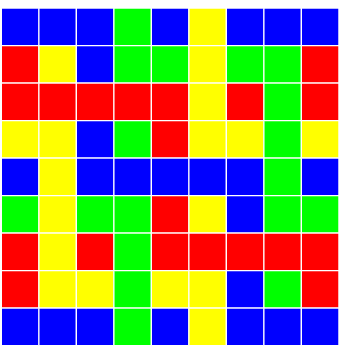}\\%30%p.7
\mbox{(a)} &\mbox{(b)}
\end{array}\]
\caption{a. Roth's \cite{Roth1993} species-$23_e$ example 8-27-1. b. Four-colouring with redundant cells along mirrors.}\label{fig:55a}
\end{figure}

For species  $11_e$, $13_e$, $15_e$, $17_e$, $18_s$, $19_e$, 21, 22, $23_e$, $25_e$, and $26_e$, closest-together axes or mirrors ($17_e$, $18_s$, 21, $23_e$ mirrors; $19_e$, 22 axes) with separation a multiple $m\geq 1$ of $2\delta$ can be placed along dark lines' mirrors.
The fabric's perpendicular axes or mirrors will fall on mirrors perpendicular to the dark lines.
Examples are the fabric of Figure~\ref{fig:40a}b of species $11_e$ coloured in Figure~\ref{fig:50a}a,  Roth's \cite{Roth1993} species-$17_e$ example in Figure~\ref{fig:50b}, a fabric of species $18_s$ in Figure~\ref{fig:51a}, Roth's \cite{Roth1993} species-$19_e$ example in Figure~\ref{fig:52a}, Roth's \cite{Roth1993} species-21 example in Figure~\ref{fig:53a}, the fabric of species 22 in Figure~\ref{fig:20a}a coloured in Figure~\ref{fig:54}a, Roth's \cite{Roth1993} species-$23_e$ example in Figure~\ref{fig:55a}, the fabric of species $25_e$ in Figure~\ref{fig:22a}a coloured in Figure~\ref{fig:50a}b, and the fabric of species $26_e$ in Figure~\ref{fig:27a}a coloured in Figure~\ref{fig:54}b.

\begin{figure}%48 not 43
\[\begin{array}{cc}
\epsffile{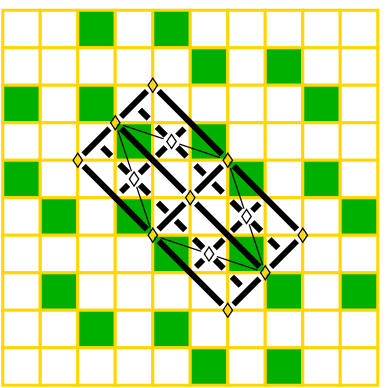} &\epsffile{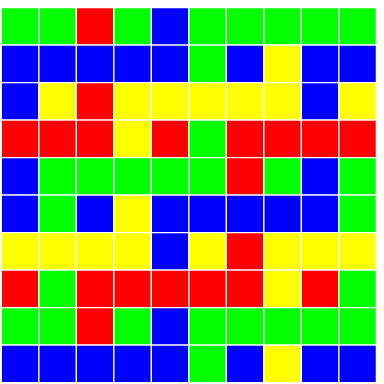}\\%30%perp 13b p.9 (file 12b) for 28n (Cf. 23e.)
\mbox{(a)} &\mbox{(b)}
\end{array}\]
\caption{a. Fabric 8-5-1 of species $28_n$ (Fig.~13b of \cite{Thomas2010a}). b. Four-colouring with redundant cells along mirrors of positive slope}\label{fig:58a}.
\end{figure}
\begin{figure}%49 not 44
\epsffile{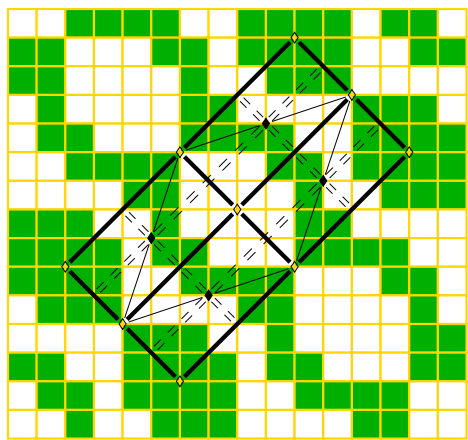}%30%perp 15a p.9(file 14a) for 29
\caption{Roth's \cite{Roth1993} example 16-2499 of species $29$ (Fig.~15a of \cite{Thomas2010a}).}\label{fig:59a}
\end{figure}
\begin{figure}%50 not 45
\[\begin{array}{cc}
\epsffile{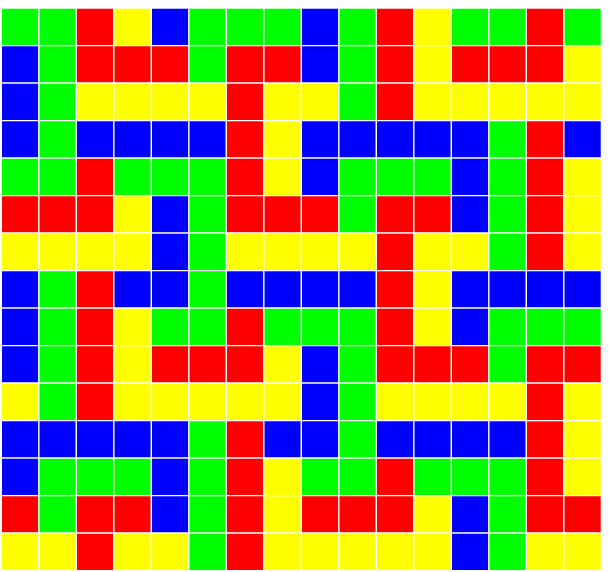} &\epsffile{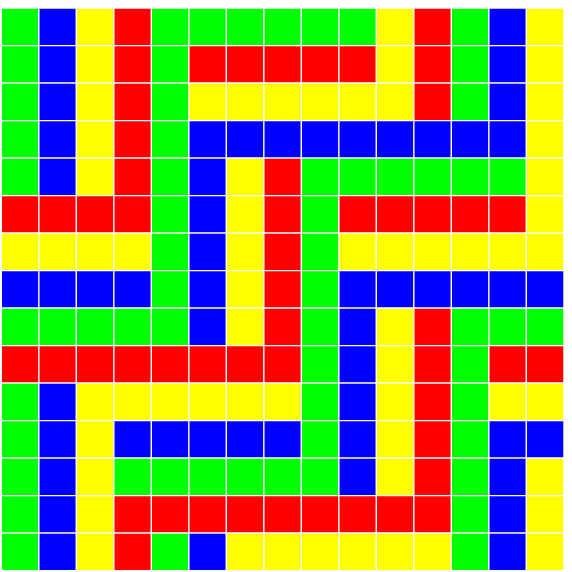}\\%30%a. perp 15b p. 9 (Fig 40a---file44a) for 3 
\mbox{(a)} &\mbox{(b)}
\end{array}\]
\caption{a. Four-colouring of the fabric of Figure~\ref{fig:59a} with redundant cells along mirrors of negative slope. b. Four-colouring of the fabric of Figure~\ref{fig:44a} with redundant cells along mirrors of positive slope.}\label{fig:60}
\end{figure}
Species  $27_e$, $28_e$, $28_n$, 29 with the spacing of $27_e$, and 31 with the spacing of $28_e$, can have their mirrors with separation a multiple $m\geq 1$ of $2\delta$ placed along the dark lines' mirrors.
Intervening fabric axes will then lie along intervening mirrors, and perpendicular fabric axes and mirrors will lie along perpendicular mirrors.
Examples are fabric 8-5-1 of species $28_n$ in Figure~\ref{fig:58a} (compare Fig.~\ref{fig:55a}b), Roth's \cite{Roth1993} example of species 29 in Figure~\ref{fig:59a} coloured in Figure~\ref{fig:60}a, and the species-31 fabric of Figure~\ref{fig:44a}a coloured in Figure~\ref{fig:60}b.

The above discussion for four colours contains the expression `multiple of $2\delta$' repeatedly.
For $4p$ colours, $2\delta$ needs to be replaced by $2p\delta$.
\begin{Thm}
In each of the species  $1$--$3$, $5$--$9$, $11$, $13$, $15$, $17$, $18$, $19$, $21$--$23$, $25$--$29$, and $31$, that is, all those allowed by the mirror-position requirement on glide-reflection axes and the quarter-turn ban except $30$, there are fabrics that can be perfectly coloured by thin striping with $4p$ colours and redundant cells arranged as a $(4p-1)/1$ twill for $p=1, 2, \dots$.
\label{thm:7}
\end{Thm}

\section{Colouring woven flat tori}%
\label{sect:tori}

\noindent The weaving and colouring by thick striping of flat tori was discussed in \S\S 2 and 7 of \cite{Thomas2013}. The possibility of weaving Klein bottles is mentioned there but not discussed. The opposite sides of a period parallelogram of a design or pattern can be identified to weave or colour a woven torus, e.g., a quarter of Figure~\ref{fig:8dea}a. The rectangular lattice unit in the same figure or a square multiple of the order-by-order square, like the whole of Figure~\ref{fig:8dea}a, can also be used. In the latter case, there are $mn$ strands in each direction if there are $m$ by $m$ $n\times n$ squares, but in the former case there are usually rather fewer, in this case one strand in each direction. Too few strands are obviously an obstacle to interesting colouring.

As redundancy of cells plays no role in choosing period parallelograms in the plane to map to tori, the patterns with no redundancy discussed in section~\ref{sect:noRed} behave, if their stripes are thick, like those with thick stripes discussed in \S 7 of \cite{Thomas2013} --- another reason why section~\ref{sect:noRed} should have been in \cite{Thomas2013}. For the same reason, if their stripes are thin, they behave like those with thin stripes and twilly redundancy, to which we now turn.

The designs of the fabrics coloured in sections~\ref{sect:noRed} and \ref{sect:3C}--\ref{sect:4C} can all be used to weave flat tori. For colouring the torus, the topological constraint `that the rectangle to be mapped to the torus must be a period parallelogram' has in addition `that the rectangle must also be a period parallelogram of the colouring'.

This topic was ignored altogether in the papers \cite{Thomas2011} and \cite{Thomas2012} on striping with two colours. In that case the parities of the dimensions of lattice units or of the smallest oblique rectangle (at $\pi/4$) enclosing a rhombic lattice unit (e.g., Figure~\ref{fig:14a}) indicate whether they are suitable for two-colouring a torus with thin striping. The species at issue, i.e., those that can be two-coloured in the plane are, according to Theorems 2.1 and 2.2 of \cite{Thomas2011}, $1_m$, $2_m$, 3 (Fig.~5a of \cite{Thomas2011}), $5_o$, $5_e$, 6 (Fig.~5b of \cite{Thomas2011}), 7, $8_e$ (Fig.~5c of \cite{Thomas2011}), 9 (Fig.~5d of \cite{Thomas2011}), 11 (Fig.~9ab of \cite{Thomas2011}), 13, 15, 17, $18_s$, 19, 21, 22 (Fig.~9cd of \cite{Thomas2011}), 23, 25--27, $28_e$, $28_n$, 29, 30 (Fig.~9ef of \cite{Thomas2011}), 31, and $36_s$. 
Of these, most allow the lattice unit of the striped pattern to be used for a torus. There are two kinds of exception. Patterns of species $1_m$, $2_m$ (Fig.~4ab of \cite{Thomas2011}), $5_e$ (Fig.~4ef of \cite{Thomas2011}), 7 (Fig.~6 of \cite{Thomas2011}), $18_s$, and $26_e$ need to have one of the odd lattice-unit dimensions doubled, and patterns of species $5_o$ (Fig.~4cd of \cite{Thomas2011}) and $26_o$ must have both doubled. As well, the lattice unit for species $36_s$ is interesting because, while its square lattice unit can simply be doubled in both directions, taking it from level 2 to level 4, the square at the intermediate level 3 is sufficient to make the weaving work for a torus, as can be seen in Figures 7 and 8 of \cite{Thomas2011}. 

Order-by-order squares are suitable in all cases because all perfectly 2-colourable fabrics are of even order.

When we turn to more than two colours, we find that everything is as for thick striping but simpler. Given a perfect $c$-colouring by thin striping of an order-$n$ isonemal design, one needs to choose a suitable period parallelogram. Since the design is given, with its various period parallelograms, both $n$-by-$n$ and oblique lattice units of the side-preserving subgroup $H_1$, it is just a question of how many copies are needed. In the  $n$-by-$n$ case, the smallest $mn$-by-$mn$ choice must have $mn$ be the lowest common multiple of $n$ and $c$, since each stripe of the colouring is one cell wide. The size of a lattice unit must be inflated (if at all) similarly. The smallest period parallelograms of any colouring can of course be assembled to make others as large as one pleases.

Another approach would take the number of colours $c$ and the species as given and select the lattice unit appropriate to the species, if any, so that fabrics based on it are perfectly colourable.

Consider the example of Figure~\ref{fig:9ab}, a 3-coloured species-$1_o$ fabric of order $n=30$. The lattice unit shown in Figure~\ref{fig:9ab}a is too small to become a torus with the colouring of Fig.~\ref{fig:9ab}b because it is $3\delta \times 5\delta$. While the pattern repeats with the same colour with a translation of $3\delta$ in the direction perpendicular to the axes, the illustrated lattice unit has to be repeated three times along the axes to bring the colours as well as the pattern into phase, making the torus $3\delta \times 15\delta$. In this case, the $mn$-by-$mn$ square of the first paragraph of the section would be 30 by 30 with area 900 since $m$ can be 1, as against $3\delta \times 15\delta =90$ for the oblique rectangle. More interestingly, the $30\times 30$ square/torus has 10 thin stripes of each colour in each direction, whereas the rectangle/torus has only a single strand of each colour in each direction, each crossing the torus six times on account of the way the strand's ends in the rectangular lattice unit match up in the torus.

Because a woven flat torus is a link in the sense of knot theory, it is of interest how many strands there are --- the same number of each colour in each direction. As an example of what happens, let us pursue the example of the previous paragraph. When the $3\delta\times 5\delta$ $H_1$ lattice unit is replicated to make the $3\delta\times 15\delta$ torus, a translational symmetry subgroup is created, for the successive images of the $3\delta\times 5\delta$ lattice unit, while woven identically, must differ by a cyclic permutation of the colours. The behaviour now being discussed is independent of how the torus is woven but depends solely on the dimensions of the $H_1$ lattice unit. Considering vertical strands and starting with the strand next to the bottom corner of the lattice unit marked in Figure~\ref{fig:9ab}a, they are yellow, red, and blue cyclically left to right. When these three strands, which we can call a tricoloured \emph{band} reaches the upper boundary of the lattice unit, they are identified with the next band to the right by the identification of boundaries, and then the next, and so on until eventually they emerge at the upper right boundary of the third copy of the lattice unit to be identified with the band entering the lower left boundary upward. This is identified with the band with which we began. There is a single band in each direction and so a single strand of each colour.

The multiple-of-three dimensions required for torus formation with colours in phase ensure that the bands of the previous paragraph can be chosen as units that will not be split up as they pass across the torus. The twilly redundancy of the colouring ensures that the colouring of the bands in one direction is a cyclic permutation of the colouring of the strands of the bands in the other direction, but there is no reason for them to be the same.

The $3\delta\times 15\delta$ torus is by no means the only possibility with this lattice unit. The short dimension can be multiplied by any positive integer. Five gives an interesting result for the $15\delta\times 15\delta$ square torus. The five bands that enter the lower right side of the period parallelogram (= the torus) reach its upper right side together, are identified with the five bands entering the lower left side, and rise to the upper left side to be identified with themselves respectively. As happens with any square period parallelogram at $\pi/4$ to the horizontal, the bands --- and so all the strands --- remain separate. There are five bands of each colour in each direction. Any multiple of five gives the same result.

What about non-multiples of five? If the $3\delta\times 15\delta$ period parallelogram is multiplied by 7 to be $21\delta\times 15\delta$, then as the five bands at the lower right boundary rise up through copies of the lattice unit, we know that, as they cross the boundary between the fifth and sixth copies of the lattice unit, they fill it as they did the lower right side --- as they did in the previous paragraph. The multiples like 7 can be taken \emph{modulo} 5 --- or whatever makes a square. Seven gives the same result as two. Crossing only two of the $3\delta\times 15\delta$ rectangles brings the fourth and fifth bands to the left end of the upper left side, with the first, second, and third bands following on. As the permutation
$\left(
\begin{smallmatrix}
1 &2 &3 &4 &5\\
4 &5 &1 &2 &3
\end{smallmatrix}
\right)$
has period 5 like
$\left(
\begin{smallmatrix}
1 &2 &3 &4 &5\\
5 &1 &2 &3 &4
\end{smallmatrix}
\right)$
for just one $3\delta\times 15\delta$, each supposed band passes through each position on its successive crossings of the torus, and so there is actually only one band. The same happens for multiples congruent to 3 and 4 \emph{modulo} 5, making congruence to zero quite special. The primality of 5, the apparent number of bands entering the lower right side, makes one band be the almost uniform answer.

When the apparent $b$ bands entering one side of the smallest period parallelogram, which is an $H_1$ lattice unit perhaps replicated to put the colours in phase, reach the opposite side of the period parallelogram they are permuted by running into a perpendicular side. Call that permutation $p$. A square period parallelogram has actually $b$ distinct bands, the index of the identity in the cyclic subgroup of permutations, in the example $\left(
\begin{smallmatrix}
1 &2 &3 &4 &5\\
5 &1 &2 &3 &4
\end{smallmatrix}
\right)^5$. When $b$ is prime, we have seen that, because the index of the powers of $p$ other than $p^b$ is one, there is only one band. When $b$ is composite there is room for different indices and therefore numbers of bands different from 1 and $b$.

\nocite{*}
\bibliographystyle{amsplain}

\begin{thebibliography}{22}

\bibitem[1]{GS1980} Gr\"unbaum, Branko, and Geoffrey C.~Shephard, \emph{Satins and twills: An introduction to the geometry of fabrics}, Mathematics Magazine \textbf{53} (1980), 139--161.

\bibitem[2]{GS1985} Gr\"unbaum, Branko, and Geoffrey C.~Shephard, \emph{A catalogue of isonemal fabrics}, in \emph{Discrete Geometry and Convexity}, Jacob E.~Goodman \emph{et al.,} eds. Annals of the New York Academy of Sciences \textbf{440} (1985), 279--298.%

\bibitem[3]{GS1986} Gr\"unbaum, Branko, and Geoffrey C.~Shephard, \emph{An extension to the catalogue of isonemal fabrics}, Discrete Mathematics \textbf{60} (1986), 155--192.

\bibitem[4]{GS1988} Gr\"unbaum, Branko, and Geoffrey C.~Shephard, \emph{Isonemal fabrics}, American Mathematical Monthly \textbf{95} (1988), 5--30.%

\bibitem[5]{Roth1993} Roth, Richard L., \emph{The symmetry groups of periodic isonemal fabrics}, Geometriae Dedicata \textbf{48} (1993), 191--210.

\bibitem[6]{Roth1995} Roth, Richard L., \emph{Perfect colorings of isonemal fabrics using two colors}, Geometriae Dedicata \textbf{56} (1995), 307--326.

\bibitem[7]{Thomas2009} Thomas, R.S.D., \emph{Isonemal prefabrics with only parallel axes of sym\-metry}, Discrete Mathematics \textbf{309} (2009), 2696--2711. doi:10.1016/j.disc.2008.06.028. Online: \url{http://arxiv.org/abs/math/0612808v2} and on the journal site.

\bibitem[8]{Thomas2010a} Thomas, R.S.D., \emph{Isonemal prefabrics with perpendicular axes of symmetry}, Utilitas Mathematica \textbf{82} (2010), 33--70. Online: \url{http://arxiv.org/abs/0805.3791v1}.

\bibitem[9]{Thomas2010b} Thomas, R.S.D., \emph{Isonemal prefabrics with no axes of symmetry}, Discrete Mathematics \textbf{310} (2010), 1307--1324. doi:10.1016/j.disc.2009.12.015. Online: \url{http://arxiv.org/abs/0911.1467v2} and on the journal site.

\bibitem[10]{Thomas2011} Thomas, R.S.D., \emph{Perfect colourings of isonemal fabrics by thin striping}, Bull.~Australian Math.~Soc. \textbf{83} (2011), 63--86. 
Online: \url{http://arxiv.org/abs/1006.5653} and on the journal site, doi:10.1017/S0004972710001632. 

\bibitem[11]{Thomas2012} Thomas, R.S.D., \emph{Perfect colourings of isonemal fabrics by thick striping}, Bull.~Australian Math.~Soc. \textbf{85} (2012), 325--349. 
Online: \url{http://arxiv.org/abs/1109.2254} and on the journal site, doi:10.1017/S0004972711002899. 

\bibitem[12]{Thomas2013} Thomas, R.S.D., \emph{Colouring isonemal fabrics with more than two colours by thick striping}, Contributions to Discrete Mathematics {\bf 8} (2013), 38--65. Online: \url{http://cdm.math.ucalgary.ca/cdm/index.php/cdm/article/view/365/165}.
\end{thebibliography}

\end{document}